\documentclass[11pt]{article}
\usepackage[T1]{fontenc}
\usepackage[utf8]{inputenc}
\usepackage{lmodern,amsmath,amssymb,mathrsfs,amsthm,mathtools,geometry}
\usepackage[hidelinks]{hyperref}
\newtheorem{theorem}{Theorem}[section]
\newtheorem{proposition}[theorem]{Proposition}
\newtheorem{lemma}[theorem]{Lemma}
\newtheorem{corollary}[theorem]{Corollary}
\newtheorem{example}[theorem]{Example}

\newcommand{\norm}[1]{\lVert#1\rVert}
\title{Finite-Branch Multipliers on a H\"older Algebra:\\ Commutants, Norm Closures, and Essential Spectra}
\author{H\"ulya Duru \and Serkan \.{I}lter \and Arkady Kitover}
\date{}
\begin{document}
\maketitle
\begin{abstract}
Noninjective multiplication symbols allow commuting operators to mix
function values along their fibers, while H\"older regularity constrains
this mixing at branch points. We study these operators on the Banach
algebra $A_\beta$, $0<\beta<1$, defined by a restricted interval-variation condition.
For $u(t)=t(1-t)$, we determine the full commutant
$\{M_u\}'=\mathscr C_\beta$, including operators whose reflection
coefficients need not belong to $A_\beta$. Separately, we characterize the operator-norm closure $J_\beta$
of the regular multiplication--reflection subalgebra of $\mathscr C_\beta$,
obtain two-sided norm-distance estimates, and find a complemented copy of
$\ell^\infty/c_0$ in the Banach-space quotient
$\mathscr C_\beta/J_\beta$. Every
$T\in\mathscr C_\beta$ satisfies
$\sigma(T)=\sigma_{\mathrm{ap}}(T)=\sigma_{\mathrm{usf}}(T)
=\sigma_{\mathrm{ess}}(T)$, and Fredholmness is equivalent to
invertibility. For multipliers with finitely many full bi-Lipschitz
branches, we also characterize their full commutants and
establish the same spectral equalities for every commuting operator.
For the regular finite-branch algebra, we establish an exact
norm-closure criterion using weighted local H\"older defects and
separate row-difference conditions, even when several folds share a level.
\end{abstract}

\noindent\textbf{Keywords:} H\"older algebra with restricted variation; commutants of multiplication operators; multiplication--reflection operators; finite-branch multipliers; operator-norm closure; essential spectra; Fredholm operators.\par
\smallskip
\noindent\textbf{2020 Mathematics Subject Classification:} Primary 47B48, 47B38; Secondary 47A10, 47A53, 46J10.

\section*{Introduction}
When a multiplication symbol takes the same value at different
points, a commuting operator may combine function values on the same
fiber. On a H\"older function algebra, this combination must also
respect regularity where inverse branches meet. We investigate three
questions: which bounded operators commute with a noninjective
multiplier, which of them can be approximated in operator norm by
regular-coefficient operators, and how their spectra reflect the
fiber structure. We begin with the symmetric symbol $u(t)=t(1-t)$
and then consider finitely many full branches.

Our space $A_\beta$, $0<\beta<1$, is a closed subalgebra of
$C^\beta([0,1])$. Its restricted-variation condition requires sums
of absolute increments over disjoint intervals to vanish as their
total length tends to zero, under a separate fixed bound on the sum
of the $\beta$-powers of the interval lengths. This condition yields the strict
inclusions $c^\beta\subsetneq A_\beta\subsetneq C^\beta$.
Unlike the little H\"older space, $A_\beta$ contains the cusp $|t-\tfrac12|^\beta$ and the odd function
$h(t)=\operatorname{sgn}(t-\tfrac12)|t-\tfrac12|^\beta$, which is
central to coefficient recovery; unlike the full H\"older space,
it excludes certain small-scale oscillations. Some of our localization arguments also work on $C^\beta$, and we show exactly which operators in the ordinary H\"older commutant preserve $A_\beta$, using only the test functions $1$ and $h$.

Several methods we use have appeared in earlier work. Sherbert
\cite{Sherbert1963,Sherbert1964} studied Lipschitz function algebras.
Khani Robati and Seddighi \cite{KhaniRobatiSeddighi2000} treated
operator-matrix commutants of direct sums and multipliers of
Lipschitz algebras. Results for quadratic and other analytic
multiplication symbols appear in
\cite{KhaniRobati2000Square,KhaniRobatiVaezpour2001,AhmadiHedayatian2013,YarmohammadiKhaniRobati2025};
related Lipschitz operator classes and finite Blaschke-product
commutants are studied in
\cite{GolbaharanMahyar2016,GallardoPartington2022,ChattopadhyayDas2024}.
Distance formulas for oscillation and Lipschitz--H\"older function
spaces appear in \cite{Perfekt2013,Angrisani2020};
\cite{PlichkoWojtowicz2007} concerns Banach-space quotients, while
\cite{ReinwandKasprzak2023,KitoverOrhon2025,BehrouziMahyar2025}
treat spectral and Fredholm questions. The general use of finite-dimensional fiber matrices, cutoff and localization arguments, and quotient-space distance formulas is not new, and we do not claim that the two-budget function-space condition defining $A_\beta$ was first introduced here. Our contribution consists of the exact algebraic and endpoint conditions, the quantitative closure criteria, and the spectral conclusions in the specified restricted-variation setting.

For $u(t)=t(1-t)$, our first main result is
\[
  \{M_u\}'=\mathscr C_\beta.
\]
The algebra $\mathscr C_\beta$ consists of bounded
multiplication--reflection operators with uniquely determined
coefficients. Their reflection coefficients may fail to belong to
$A_\beta$, so the regular-coefficient algebra does not give the
entire commutant. Exact finite-fiber locality, proved before the
commutant theorem, gives the representation for every
$f\in A_\beta$ without relying on polynomial density. For symbols
with finitely many full bi-Lipschitz branches, we characterize
\emph{every} bounded commuting operator by weighted matrix
conditions at the branch endpoints. We also determine how the
symmetric commutant sits inside the ordinary H\"older commutant.

Our second main result characterizes the operator-norm closure
$J_\beta$ of the regular multiplication--reflection algebra by a
local H\"older defect, with explicit two-sided distance estimates.
The regular algebra, $J_\beta$, and $\mathscr C_\beta$ are distinct;
moreover, the Banach-space quotient $\mathscr C_\beta/J_\beta$
contains a complemented copy of $\ell^\infty/c_0$. We obtain
corresponding complete closure and distance criteria for regular
finite-branch operators. We also identify the quotient of the full finite-branch commutant by this closure, up to isomorphism, with a closed subspace of a sum of two local germ quotients. At simultaneous folds, compatibility with
the common endpoint subspace alone is insufficient: the row
difference associated with each fold must lie in its own difference
direction. A four-branch example exhibits this obstruction.

Our third main result concerns the entire commutant rather than
only the norm closure of regular operators. If $T$ commutes with
the symmetric multiplier, or with a multiplier having finitely
many full bi-Lipschitz branches, then
\[
 \sigma_{\mathrm{ap}}(T)=\sigma_{\mathrm{usf}}(T)
 =\sigma_{\mathrm{ess}}(T)=\sigma(T).
\]
For these classes, $T$ is Fredholm if and only if it is invertible.
The spectrum is recovered from finite-dimensional fiber matrices,
with the required closure of their interior spectral values.
The proofs combine localized approximate eigenvectors with
weighted inverse constructions that preserve the endpoint
conditions. We do not extend these spectral assertions to arbitrary
finite-to-one symbols without the stated full-branch hypotheses.

Section~1 establishes the function algebra; Section~2 proves
finite-fiber locality and the multiplier criterion. Sections~3--4
give the full symmetric commutant, its norm closure, and its spectral
and Fredholm properties. Sections~5--6 treat regular finite-branch
operators and fold defects. Section~7 identifies the full
finite-branch commutant and its spectra.

For the notation used below, we write $\mathscr D_\beta$ for regular multiplication--reflection
operators, $\mathscr C_\beta$ for the closed symmetric commuting
algebra, and $J_\beta$ for the operator-norm closure of
$\mathscr D_\beta$ within $\mathscr C_\beta$. The regular
$n$-branch algebra is $\mathscr R_{\beta,n}$, while
$\mathfrak M_{\beta,n}$ denotes the larger class of bounded
fiber-matrix operators satisfying (S1)--(S2); by Corollary~\ref{cor:class-full}, it is exactly the full commutant under the branch isomorphism. Their definitions,
and the endpoint spaces $W_e$ and defect $\Omega_e$, are given
in the sections where they are used. Throughout, $\sigma(T)$ denotes the spectrum of $T$ in the algebra of all bounded operators on the space on which $T$ acts.

\section{Preliminaries and notation}
Throughout, $0<\beta<1$, and functions are complex-valued on $[0,1]$. Write
\[
[f]_\beta=\sup_{s\ne t}\frac{|f(t)-f(s)|}{|t-s|^\beta},\qquad \norm{f}_\beta=\norm{f}_\infty+[f]_\beta.
\]

For a bounded operator $S$ on a complex Banach space, let
$\sigma(S)$ and $\sigma_{\mathrm{ap}}(S)$ denote its spectrum and
approximate point spectrum. An operator is \emph{upper semi-Fredholm}
if its kernel is finite-dimensional and its range is closed, and it is
\emph{Fredholm} if, in addition, its range has finite codimension.
We write $\sigma_{\mathrm{usf}}(S)$ for the set of $\lambda$ such that
$S-\lambda I$ is not upper semi-Fredholm and
$\sigma_{\mathrm{ess}}(S)$ for the set of $\lambda$ such that
$S-\lambda I$ is not Fredholm. Thus we use the Fredholm definition
of essential spectrum throughout.
Let $C^\beta([0,1])$ be the associated Banach algebra. For $\delta,M>0$, let $V_{\beta,\delta,M}(f)$ be the supremum of $\sum_{j=1}^n|f(b_j)-f(a_j)|$ over finite families of intervals $[a_j,b_j]\subseteq[0,1]$ with pairwise disjoint interiors satisfying
\[
\sum_j(b_j-a_j)\le\delta,\qquad\sum_j(b_j-a_j)^\beta\le M.
\]
Define
\[
A_\beta=\left\{f\in C^\beta([0,1]):\lim_{\delta\downarrow0}V_{\beta,\delta,M}(f)=0\text{ for every }M>0\right\}.
\]
\noindent
The defining supremum involves ordinary increments of $f$ over
intervals, with a total-length constraint and a separate $\beta$-power
budget. The dependence on $\beta$ is twofold: in the ambient
H\"older norm and in the admissible interval families. For a closed subinterval $J\subseteq[0,1]$, $A_\beta(J)$ denotes the algebra defined by the same condition on $J$ with its natural norm; $A_\beta(J;\mathbb C^n)$ and $A_\beta(J;M_n(\mathbb C))$ denote the corresponding vector- and matrix-valued spaces.

\begin{proposition}\label{prop:banach}
$A_\beta$ is a unital Banach algebra under pointwise multiplication.
\end{proposition}
\begin{proof}
The H\"older product estimate gives $\norm{fg}_\beta\le\norm f_\beta\norm g_\beta$. For every admissible family,
\[
V_{\beta,\delta,M}(fg)\le\norm f_\infty V_{\beta,\delta,M}(g)+\norm g_\infty V_{\beta,\delta,M}(f).
\]
Thus $A_\beta$ is closed under products. If $f_k\to f$ in $C^\beta$, then
\[
V_{\beta,\delta,M}(f)\le V_{\beta,\delta,M}(f_k)+M[f-f_k]_\beta.
\]
First choose $k$ large and then let $\delta\downarrow0$; this proves that $A_\beta$ is closed in $C^\beta$. The constant function $1$ belongs to $A_\beta$.
\end{proof}
\begin{lemma}\label{lem:ac}
$AC([0,1])\cap C^\beta([0,1])\subseteq A_\beta$.
\end{lemma}
\begin{proof}
For any admissible family, $\sum_j|f(b_j)-f(a_j)|\le\sum_j\int_{a_j}^{b_j}|f'(t)|\,dt$. The intervals have pairwise disjoint interiors, so the integrals sum to the integral over their union, whose Lebesgue measure is at most $\delta$. Absolute continuity of the Lebesgue integral gives the required limit.
\end{proof}
\begin{proposition}\label{prop:strict-holder-inclusions}
Let $c^\beta([0,1])$ denote the little H\"older space, consisting of all $f\in C^\beta([0,1])$ such that
\[
\lim_{r\downarrow0}\sup_{0<|t-s|\le r}\frac{|f(t)-f(s)|}{|t-s|^\beta}=0.
\]
Then both inclusions in
\[
c^\beta([0,1])\subsetneq A_\beta\subsetneq C^\beta([0,1])
\]
are strict.
\end{proposition}
\begin{proof}
For $f\in c^\beta$, put $\omega_\beta(f,r)=\sup_{0<|t-s|\le r}|f(t)-f(s)|/|t-s|^\beta$. Every admissible interval has length at most $\delta$, so
$V_{\beta,\delta,M}(f)\le M\omega_\beta(f,\delta)\to0$.
The function $t\mapsto|t-1/2|^\beta$ belongs to $AC\cap C^\beta$, hence to $A_\beta$ by Lemma~\ref{lem:ac}, but it is not little H\"older.

To prove the second inclusion strict, choose pairwise disjoint closed intervals $I_n\subset(0,1)$ of lengths $L_n\downarrow0$, and integers
\[
m_n=\left\lceil L_n^{-\beta/(1-\beta)}\right\rceil,\qquad w_n=L_n/m_n.
\]
Partition each $I_n$ into $m_n$ cells $[a,a+w_n]$. On each cell set
$f(t)=\min\{t-a,a+w_n-t\}^\beta$, and set $f=0$ outside the cells. The resulting function is continuous: its cell heights are at most $(w_n/2)^\beta\to0$. Moreover, if $s,t$ lie in different cells, a cell endpoint between them and the vanishing of $f$ there give
$|f(t)-f(s)|\le 2^{1-\beta}|t-s|^\beta$; within a single cell the same estimate holds. Thus $f\in C^\beta$.
For each $n$, take the left half of every cell in $I_n$. Their total length is $L_n/2\to0$, while both the sum of their $\beta$-powers and the sum of the corresponding increments of $f$ equal
\[
2^{-\beta}L_n^\beta m_n^{1-\beta}.
\]
Since $1\le L_n^\beta m_n^{1-\beta}\le (1+L_n^{\beta/(1-\beta)})^{1-\beta}$, these sums stay between two positive constants. In particular, the families are admissible for one fixed $M>0$, and $V_{\beta,L_n/2,M}(f)\ge2^{-\beta}$ for every $n$. Hence $f\notin A_\beta$.
\end{proof}
The two strict inclusions explain the choice of $A_\beta$ in this
operator problem. The cusp $|t-c|^\beta$ belongs to $A_\beta$ even
though its local H\"older quotient does not vanish, whereas the
preceding family of small tents belongs to $C^\beta$ but fails the
two-budget increment condition. In particular, $A_\beta$ provides more
local structure than $C^\beta$ without imposing the little H\"older
condition; the exact locality theorem below applies to all of $A_\beta$.

Put $c=1/2$, $u(t)=t(1-t)$, and $(Rf)(t)=f(1-t)$. Reflection is an isometric involution on $A_\beta$: reflecting each admissible interval family preserves both length constraints and the sum of absolute increments. Write $E=\{f:Rf=f\}$ and $O=\{f:Rf=-f\}$.
\begin{proposition}\label{prop:evenodd}
The operators $P_E=(I+R)/2$ and $P_O=(I-R)/2$ are bounded projections onto $E$ and $O$. In particular $A_\beta=E\oplus O$, and every $g\in O$ satisfies $g(c)=0$.
\end{proposition}
\begin{proof}
Use $R^2=I$ and $\norm R=1$. If $g\in O$, then $g(c)=-g(c)$.
\end{proof}
Let $P_u$ be the closure in $A_\beta$ of the polynomials in $u$.
\begin{proposition}\label{prop:poly}
$P_u\subsetneq E$.
\end{proposition}
\begin{proof}
Since $Ru=u$, $P_u\subseteq E$. Let $v(t)=|t-c|^\beta$. This function is symmetric, absolutely continuous and $\beta$-H\"older, so $v\in E$ by Lemma~\ref{lem:ac}. Every $p(u)$ is continuously differentiable, hence
\[
\lim_{t\to c}\frac{|p(u(t))-p(u(c))|}{|t-c|^\beta}=0,
\]
whereas $|v(t)-v(c)|/|t-c|^\beta=1$. Therefore $[v-p(u)]_\beta\ge1$ for every polynomial $p$ and $v\notin P_u$.
\end{proof}

\medskip
The following elementary composition property is used repeatedly below, in particular to form reciprocals.
\begin{lemma}\label{lem:composition-reciprocal}
Let $J\subseteq[0,1]$ be a closed interval, let $f\in A_\beta(J)$, and let $\vartheta:K\to\mathbb C$ be Lipschitz with constant $L$ on a set $K\subseteq\mathbb C$ containing $f(J)$. Then $\vartheta\circ f\in A_\beta(J)$, with
\[
 [\vartheta\circ f]_{\beta,J}\le L[f]_{\beta,J},\qquad
 V_{\beta,\delta,M}(\vartheta\circ f)\le L V_{\beta,\delta,M}(f).
\]
In particular, if $f$ has no zeros on $J$, then $1/f\in A_\beta(J)$; the same holds entrywise for the inverse of a matrix function in $M_n(A_\beta(J))$ whose determinant has no zeros on $J$.
\end{lemma}
\begin{proof}
For all $s,t\in J$, $|\vartheta(f(t))-\vartheta(f(s))|\le L|f(t)-f(s)|$; this gives both estimates, and $\vartheta\circ f$ is bounded since $f(J)$ is compact. If $f$ has no zeros, then $|f|\ge m>0$ on $J$ and $z\mapsto1/z$ is $m^{-2}$-Lipschitz on $\{|z|\ge m\}$. For matrices apply this to the determinant and use the adjugate formula and Proposition~\ref{prop:banach}.
\end{proof}

\section{Exact finite-fiber locality and the multiplier criterion}
Write $A_{\beta,0}=\{b\in A_\beta:b(c)=0\}$ and
\[
h(t)=\operatorname{sgn}(t-c)|t-c|^\beta\quad(t\ne c),\qquad h(c)=0.
\]
The function $h$ belongs to $A_{\beta,0}$ by Lemma~\ref{lem:ac}.

\begin{lemma}\label{lem:uniform-localization}
Let $F\subset[0,1]$ be nonempty and finite. If $g\in A_\beta$ vanishes on $F$, then for every $r>0$ there is $g_r\in A_\beta$ such that
\[
 g_r=g\ \text{on }\{t:\operatorname{dist}(t,F)\le r\},\qquad
 \operatorname{supp}g_r\subset\{t:\operatorname{dist}(t,F)\le 2r\},
 \qquad \|g_r\|_\beta\le C_\beta\|g\|_\beta .
\]
The constant does not depend on $F$, $r$, or $g$.
\end{lemma}
\begin{proof}
Choose a fixed Lipschitz function $\chi:[0,\infty)\to[0,1]$ equal to one on $[0,1]$ and zero on $[2,\infty)$. Put $\chi_r(t)=\chi(\operatorname{dist}(t,F)/r)$ and $g_r=\chi_r g$. Since $\chi_r$ is Lipschitz and $A_\beta$ is an algebra containing Lipschitz functions, $g_r\in A_\beta$. On its support,
$|g(t)|\le [g]_\beta(2r)^\beta$. For $|s-t|\le r$, expand $g_r(s)-g_r(t)$ and use
$|\chi_r(s)-\chi_r(t)|\le C|s-t|/r$; when $|s-t|>r$, use the preceding supremum bound. Both cases give $[g_r]_\beta\le C_\beta[g]_\beta$; the supremum norm is bounded by $\|g\|_\infty$. This proves the assertion.
\end{proof}

\begin{theorem}\label{thm:exact-finite-fiber}
Let $u\in A_\beta$ be real valued and suppose every level set $F_t=u^{-1}(\{u(t)\})$ is finite. If $T\in\mathcal B(A_\beta)$ and $TM_u=M_uT$, then for every $t\in[0,1]$ there are unique numbers $a_{t,s}$, $s\in F_t$, such that
\begin{equation}\label{eq:exact-fiber}
 (Tf)(t)=\sum_{s\in F_t}a_{t,s}f(s)\qquad(f\in A_\beta).
\end{equation}
In particular, the conclusion holds for every $f\in A_\beta$, without a locally little-H\"older assumption. Conversely, every bounded operator given by \eqref{eq:exact-fiber} commutes with $M_u$.
\end{theorem}
\begin{proof}
Fix $t$, put $\lambda=u(t)$ and $F=F_t$. First note the usual fiber locality: if $k\in A_\beta$ vanishes on a neighborhood of $F$, then $k=(u-\lambda)q$ for some $q\in A_\beta$. Indeed, on the compact support of $k$, $|u-\lambda|$ is bounded away from zero. Multiplying a local reciprocal of $u-\lambda$ by a Lipschitz cutoff gives an element of $A_\beta$; the reciprocal is legitimate because $A_\beta$ is stable under smooth composition on the range of $u$ away from $\lambda$. Consequently $(Tk)(t)=0$.

We now pass from neighborhood fiber locality to exact locality for
\emph{every} function in $A_\beta$. Let $g\in A_\beta$ vanish on $F$.
For each $s\notin F$, choose $r=r(s)>0$ so small that
$u(\operatorname{supp}g_r)$ does not contain $u(s)$, where $g_r$ is
given by Lemma~\ref{lem:uniform-localization}. Such a choice is possible
because $u(s)\ne\lambda$, $u=\lambda$ on $F$, and
$\operatorname{supp}g_r\subset\{x:\operatorname{dist}(x,F)\le 2r\}$.
The difference $g-g_r$ vanishes on a neighborhood of $F$, while
$g_r$ vanishes on a neighborhood of $F_s$. Applying neighborhood
fiber locality at $t$ and at $s$, respectively, gives
$(Tg)(t)=(Tg_r)(t)$ and $(Tg_r)(s)=0$. Hence
\[
 |(Tg)(t)|=|(Tg_r)(t)-(Tg_r)(s)|
 \le [Tg_r]_\beta |t-s|^\beta
 \le C_\beta\|T\|\|g\|_\beta |t-s|^\beta.
\]
\emph{The uniform bound in Lemma~\ref{lem:uniform-localization} is the
key point: its constant is independent of $r=r(s)$ and of $s$. We do
not need $\|g_r\|_\beta$ to tend to zero.} Since $F$ is finite, we
can choose $s\notin F$ arbitrarily close to $t$. Letting $s\to t$
therefore yields $(Tg)(t)=0$ for every $g\in A_\beta$ with $g|_F=0$.

To obtain the formula for an arbitrary $f\in A_\beta$, let
$\ell_s$, $s\in F$, be the Lagrange interpolation polynomials on $F$,
so that $\ell_s(s')=\delta_{s,s'}$ for $s,s'\in F$. These polynomials
belong to $A_\beta$. Set $p=\sum_{s\in F}f(s)\ell_s$. Since
$f-p$ vanishes on $F$, the preceding argument gives
\[
 (Tf)(t)=(Tp)(t)
 =\sum_{s\in F}(T\ell_s)(t)f(s).
\]
Thus \eqref{eq:exact-fiber} holds for \emph{every} $f\in A_\beta$, with
$a_{t,s}=(T\ell_s)(t)$. The interpolation values are independent, so
these coefficients are unique. Conversely, $u(s)=u(t)$ for $s\in F_t$,
so \eqref{eq:exact-fiber} immediately gives $T(uf)=uTf$.
\end{proof}

\begin{corollary}\label{thm:two-point}
For $u(t)=t(1-t)$ and every bounded $T$ commuting with $M_u$,
there are unique scalars $a_T(t),b_T(t)$ for each $t\ne c$, and a unique scalar $a_T(c)$, such that
\[
 (Tf)(t)=a_T(t)f(t)+b_T(t)f(1-t)\quad(t\ne c),\qquad (Tf)(c)=a_T(c)f(c)\quad(f\in A_\beta).
\]
The evaluation map $f\mapsto(f(t),f(1-t))$ is onto $\mathbb C^2$
and has a closed kernel of codimension two.
\end{corollary}
\begin{proof}
Theorem~\ref{thm:exact-finite-fiber} gives the formula on all of
$A_\beta$. Smooth interpolation provides arbitrary independent
values at the two points, proving uniqueness, surjectivity, and the
codimension assertion. Continuity of evaluations gives closedness.
\end{proof}

\begin{lemma}\label{lem:multiplier-bound}
There is a constant $C_\beta$ such that for every $b\in A_{\beta,0}$ and $g\in O$, the function
\[
F(t)=\frac{b(t)g(t)}{h(t)}\quad(t\ne c),\qquad F(c)=0,
\]
belongs to $A_\beta$ and satisfies $\|F\|_\beta\le C_\beta\|b\|_\beta\|g\|_\beta$.
\end{lemma}
\begin{proof}
Put $B=[b]_\beta$ and $G=[g]_\beta$. Since $b(c)=g(c)=0$,
$|b(t)|\le B|t-c|^\beta$ and $|g(t)|\le G|t-c|^\beta$.
Thus $|F(t)|\le BG|t-c|^\beta$, including at $c$.

Consider $s,t$ on the same side of $c$, and write $x=\min(|s-c|,|t-c|)$,
$y=\max(|s-c|,|t-c|)$. Denote the values at the points at distances
$x$ and $y$ by subscripts $x$ and $y$. For $x>0$, expanding the
increment with the two bounded quotients gives
\[
F_y-F_x=\frac{g_y}{h_y}(b_y-b_x)
 +\frac{b_x}{h_x}(g_y-g_x)
 +b_xg_y\left(\frac1{h_y}-\frac1{h_x}\right).
\]
The first two terms are bounded in modulus by
$G|b_y-b_x|$ and $B|g_y-g_x|$, respectively. For the last term,
using that $h$ has constant sign on this side of $c$, we obtain
\[
\left|b_xg_y\left(\frac1{h_y}-\frac1{h_x}\right)\right|
\le BGx^\beta y^\beta(x^{-\beta}-y^{-\beta})
=BG(y^\beta-x^\beta).
\]
When $x=0$, the estimate follows directly from
$|F_y|\le BG y^\beta$. Since
$y^\beta-x^\beta\le(y-x)^\beta$, this proves the H\"older estimate
for points on the same side of $c$. For points on opposite sides, use
$|F(t)|\le BG|t-c|^\beta$ and
$|t-c|^\beta+|s-c|^\beta\le 2^{1-\beta}|t-s|^\beta$.
Consequently $\|F\|_\beta\le C_\beta\|b\|_\beta\|g\|_\beta$.

It remains to verify the restricted-variation condition. Let $\{[a_j,b_j]\}$
be an admissible family, and split any interval crossing $c$ at $c$.
The resulting intervals have disjoint interiors and total length at
most $\delta$; their sum of $\beta$-powers is at most
$2^{1-\beta}M$. Applying the preceding increment decomposition to
each interval on one side of $c$ gives
\[
\sum_j|F(b_j)-F(a_j)|
\le G V_{\beta,\delta,2^{1-\beta}M}(b)
 +B V_{\beta,\delta,2^{1-\beta}M}(g)
 +BG\int_U\beta|t-c|^{\beta-1}\,dt,
\]
where $U$ is the union of the split intervals (up to endpoints).
Since $|U|\le\delta$ and $\beta|t-c|^{\beta-1}$ is integrable on
$[0,1]$, absolute continuity of its integral shows that the last
term tends to zero uniformly over all admissible families as
$\delta\downarrow0$. The first two terms tend to zero by the
definition of $A_\beta$. Hence $F\in A_\beta$.
\end{proof}

\begin{theorem}\label{thm:multiplier}
Let $a$ be a function on $[0,1]\setminus\{c\}$. The following are equivalent:
\begin{enumerate}
\item[(i)] for every $g\in O$, the function $ag$, defined away from $c$, extends to an element of $A_\beta$, and the resulting map $O\to A_\beta$ is bounded;
\item[(ii)] $ah$ extends to an element of $A_{\beta,0}$.
\end{enumerate}
In that case the extension of every $ag$ takes the value zero at $c$.
\end{theorem}
\begin{proof}
(i)$\Rightarrow$(ii). Since $h\in O$, applying the multiplier to $h$ shows that $b=ah$ extends to an element of $A_\beta$. We must also prove $b(c)=0$. Put
\[
 g_0(t)=h(t)\sin\!\left(\log\frac{1}{|t-c|}\right)\quad(t\ne c),\qquad g_0(c)=0.
\]
This function is odd. On each side of $c$ its derivative satisfies $|g'_0(t)|\le(1+\beta)|t-c|^{\beta-1}$; hence $g_0$ is absolutely continuous and $\beta$-H\"older, so $g_0\in O$ by Lemma~\ref{lem:ac}. By (i), $ag_0=b(t)\sin(\log(1/|t-c|))$ extends continuously to $c$. If $b(c)\ne0$, this has no limit at $c$, a contradiction. Thus $b\in A_{\beta,0}$.

(ii)$\Rightarrow$(i). Put $b=ah\in A_{\beta,0}$. For $g\in O$ and $t\ne c$, $ag=bg/h$, so Lemma~\ref{lem:multiplier-bound} gives the extension, its vanishing at $c$, and the norm estimate.
\end{proof}

\begin{corollary}\label{cor:sign}
The coefficient $a(t)=\operatorname{sgn}(t-c)$ defines a bounded multiplier from $O$ to $A_\beta$, although it is discontinuous at $c$.
\end{corollary}
\begin{proof}
Here $ah=|t-c|^\beta\in A_{\beta,0}$ by Lemma~\ref{lem:ac}; apply Theorem~\ref{thm:multiplier}.
\end{proof}

\section{The commuting algebra, norm closure, and projections}
For $v\in A_\beta$ and $b\in A_{\beta,0}$ define
\[
(T_{v,b}f)(t)=v(t)f(t)+\frac{b(t)}{h(t)}(f(1-t)-f(t))\quad(t\ne c),
\qquad (T_{v,b}f)(c)=v(c)f(c).
\]
Put $\mathscr C_\beta=\{T_{v,b}:v\in A_\beta,\ b\in A_{\beta,0}\}$.

\begin{theorem}\label{thm:algebra}
The set $\mathscr C_\beta$ is a closed unital subalgebra of $\mathcal B(A_\beta)$ contained in the commutant of $M_u$. There are constants $m_\beta,M_\beta>0$ such that
\[
m_\beta(\|v\|_\beta+\|b\|_\beta)\le\|T_{v,b}\|\le M_\beta(\|v\|_\beta+\|b\|_\beta).
\]
For $a=b/h$ and $d=e/h$ away from $c$, multiplication is given by
\[
T_{v,b}T_{w,e}=T_{V,B},\quad
V=vw+(b/h)(Rw-w),\quad
B=ve+b(Rw)+b(Re-e)/h,
\]
where each expression is understood through its continuous $A_\beta$ extension.
\end{theorem}
\begin{proof}
Write $Qf=Rf-f\in O$. By Theorem~\ref{thm:multiplier} and
Lemma~\ref{lem:multiplier-bound},
\[
\|T_{v,b}f\|_\beta\le
\|v\|_\beta\|f\|_\beta+C_\beta\|b\|_\beta\|Qf\|_\beta
\le C'_\beta(\|v\|_\beta+\|b\|_\beta)\|f\|_\beta.
\]
Here $\|Q\|\le2$, because reflection is an isometry of $A_\beta$.
Conversely, since $T_{v,b}1=v$ and $T_{v,b}h=vh-2b$,
\[
\|v\|_\beta\le\|T_{v,b}\|,\qquad
\|b\|_\beta\le\tfrac12\bigl(\|v\|_\beta\|h\|_\beta
+\|T_{v,b}\|\|h\|_\beta\bigr)
\le\|h\|_\beta\|T_{v,b}\|.
\]
These inequalities give both norm estimates. They also show that if
$T_{v_n,b_n}$ converges in operator norm, then $(v_n)$ and $(b_n)$
are Cauchy in $A_\beta$ and $A_{\beta,0}$, respectively. Their limits
$v,b$ satisfy $T_{v_n,b_n}\to T_{v,b}$ by the upper estimate; hence
$\mathscr C_\beta$ is closed.

To verify multiplication, put $a=b/h$ and $d=e/h$ away from $c$.
For $t\ne c$, expand
\[
T_{v,b}(T_{w,e}f)(t)
=v(t)(T_{w,e}f)(t)
+a(t)\bigl((T_{w,e}f)(1-t)-(T_{w,e}f)(t)\bigr).
\]
To track the reflection sign, note that $Rh=-h$, and hence
$R(e/h)=-(Re)/h$. Writing $d=e/h$, we obtain away from $c$
\[
R(T_{w,e}f)-T_{w,e}f
=(Rw-w)f+(Rw-Rd-d)(Rf-f).
\]
Collecting the coefficient of $f(t)$ and the difference
$f(1-t)-f(t)$ therefore yields
\[
V=vw+a(Rw-w),\qquad
B=ve+b(Rw)+a(Re-e).
\]
The first coefficient belongs to $A_\beta$: $vw\in A_\beta$ and
$a(Rw-w)\in A_\beta$ by the multiplier theorem. The product
$U=T_{v,b}T_{w,e}$ commutes with $M_u$ and has the indicated two-point
formula away from $c$. Applying it to $1$ and $h$ gives
\[
U1=V,\qquad Uh=Vh-2B.
\]
Thus $B=(Vh-Uh)/2\in A_\beta$, and $B(c)=0$ because
$h(c)=0$ and $Uh(c)=0$. Consequently $B\in A_{\beta,0}$ and
$U=T_{V,B}$ also at $c$. The identity is $T_{1,0}$, and every
$T_{v,b}$ commutes with $M_u$ because $Ru=u$. Hence
$\mathscr C_\beta$ is a closed unital subalgebra of the commutant.
\end{proof}

\begin{theorem}\label{cor:full-symmetric}
For $u(t)=t(1-t)$, $c=1/2$, and $h(t)=\operatorname{sgn}(t-c)|t-c|^\beta$, one has
\begin{equation}\label{eq:full-symmetric}
 \boxed{\ \{M_u\}'=\mathscr C_\beta .\ }
\end{equation}
More precisely, every $T\in\{M_u\}'$ has the \emph{unique} representation $T=T_{v,b}$ from Section~3, with
\begin{equation}\label{eq:recovery-symmetric}
 v=T1,\qquad b=\tfrac12\bigl((T1)h-Th\bigr)\in A_{\beta,0}.
\end{equation}
\end{theorem}
\begin{proof}
Let $T$ commute with $M_u$. By Theorem~\ref{thm:exact-finite-fiber}, the value $(Tf)(t)$ for $t\ne c$ depends only on $f(t)$ and $f(1-t)$. The functions $1$ and $h$ separate these two points. Put $v=T1$ and $b=(vh-Th)/2$. Both belong to $A_\beta$. The fiber at $c$ is the singleton $\{c\}$, so Theorem~\ref{thm:exact-finite-fiber} gives $(Th)(c)=0$; thus $b(c)=0$. The multiplier theorem (Theorem~\ref{thm:multiplier}) shows that $T_{v,b}\in\mathscr C_\beta$ is bounded. Since $Rh=-h$, we have
\[
 T_{v,b}1=v=T1,\qquad T_{v,b}h=vh-2b=Th.
\]
The finite-fiber representation and separation by $1,h$ imply $(T_{v,b}f)(t)=(Tf)(t)$ for $t\ne c$, for all $f\in A_\beta$. Both sides are continuous at $c$, so equality also holds there. Uniqueness follows from \eqref{eq:recovery-symmetric}. The reverse inclusion $\mathscr C_\beta\subseteq\{M_u\}'$ was proved in Theorem~\ref{thm:algebra}.
\end{proof}

The point of the exact locality step is that the recovery from
$T1$ and $Th$ holds on \emph{all} of $A_\beta$. It does not depend on
approximating an arbitrary function by polynomials in $u$; such a
norm-density argument would be invalid by
Proposition~\ref{prop:poly}. Thus the representation is a full
commutant characterization, not just a regular-coefficient or
little-H\"older representation.

\begin{proposition}
\label{prop:holder-commutant-comparison}
Let $H_\beta=C^\beta([0,1])$ with its usual H\"older norm, and
let $u(t)=t(1-t)$ and $h(t)=\operatorname{sgn}(t-c)|t-c|^\beta$,
where $c=1/2$.
\begin{enumerate}
\item Every operator $T\in\{M_u|_{A_\beta}\}'$ has a unique
bounded extension $\widehat T\in\{M_u|_{H_\beta}\}'$.
The extension map is a continuous algebra embedding with closed,
proper range.
\item For $S\in\{M_u|_{H_\beta}\}'$, the following conditions
are equivalent:
\begin{enumerate}
\item[(i)] $S(A_\beta)\subseteq A_\beta$;
\item[(ii)] $S1,Sh\in A_\beta$;
\item[(iii)] $S=\widehat T$ for some
$T\in\{M_u|_{A_\beta}\}'$.
\end{enumerate}
\end{enumerate}
Thus preservation of restricted variation by an operator in the
ordinary H\"older commutant is determined by two test functions,
although not every such operator preserves $A_\beta$.
\end{proposition}
\begin{proof}
The H\"older part of Lemma~\ref{lem:multiplier-bound} uses only the
H\"older increment inequality. Hence, whenever $v,b\in H_\beta$
and $b(c)=0$, the formula
\[
 S_{v,b}f=vf+\frac b h(Rf-f),
\]
with its continuous value at $c$, defines a bounded operator on
$H_\beta$ commuting with $M_u$. The formula has the same meaning
on $A_\beta$ when $v,b\in A_\beta$.

The converse representation on $H_\beta$ follows
by the same exact fiber-locality argument as
Theorem~\ref{thm:exact-finite-fiber}. The uniform localization estimate
in Lemma~\ref{lem:uniform-localization} uses only H\"older estimates
and Lipschitz cutoffs, so it holds on $H_\beta$ as well; division by
$u-\lambda$ away from a level set is also valid there. Consequently,
for $S\in\{M_u|_{H_\beta}\}'$, the values of $Sf$ at $t$ depend only
on $f(t)$ and $f(1-t)$. Since $1$ and $h$ separate these two points,
$S=S_{v,b}$, where
\[
 v=S1,\qquad b=\tfrac12\bigl((S1)h-Sh\bigr),\qquad b(c)=0.
\]
In particular, the representation on $H_\beta$ is unique.

Now let $T\in\{M_u|_{A_\beta}\}'$. By
Theorem~\ref{cor:full-symmetric}, $T=T_{v,b}$ with
$v,b\in A_\beta$ and $b(c)=0$. The preceding H\"older estimate
extends this formula to a bounded commuting operator
$\widehat T=S_{v,b}$ on $H_\beta$. Uniqueness follows because any
commuting extension has the same values on $1,h$ and hence the same
coefficients.

For $S\in\{M_u|_{H_\beta}\}'$, condition (i) implies (ii)
because $1,h\in A_\beta$. Conversely, if (ii) holds, the recovered
$v,b$ belong to $A_\beta$, and
Theorem~\ref{thm:multiplier} implies that $S_{v,b}$ restricts to a
bounded operator on $A_\beta$. This proves (iii) and (i).

The coefficient estimates in Theorem~\ref{thm:algebra}, and their
H\"older-only analogues, give constants $c_\beta,C_\beta>0$ with
\[
 c_\beta\|T\|_{\mathcal B(A_\beta)}
 \le\|\widehat T\|_{\mathcal B(H_\beta)}
 \le C_\beta\|T\|_{\mathcal B(A_\beta)}.
\]
Thus the range is closed. It is proper: choose
$g\in H_\beta\setminus A_\beta$ from
Proposition~\ref{prop:strict-holder-inclusions}.
The multiplier $M_g$ is bounded on $H_\beta$ and commutes with
$M_u$, but $M_g1=g\notin A_\beta$, so it cannot preserve $A_\beta$.
\end{proof}

The comparison identifies the concrete role of restricted variation:
it imposes a genuine invariance condition on the classical H\"older
commutant, while retaining singular multipliers made possible by
$h\in A_\beta$. Indeed, $M_h$ acts on $A_\beta$ but does not
preserve $c^\beta$, since $h\notin c^\beta$ and $1\in c^\beta$.
On the other side, $M_g$ with $g\in H_\beta\setminus A_\beta$
acts on $H_\beta$ but does not preserve $A_\beta$.
Fiber localization itself also works on the classical H\"older space, as the proof above shows; restricted variation is not needed for locality.

\begin{proposition}\label{prop:center}
For $t<c$, two-point evaluation of $\mathscr C_\beta$ gives all of $M_2(\mathbb C)$. Moreover,
\[
Z(\mathscr C_\beta)=\{M_v:v\in E\}.
\]
\end{proposition}
\begin{proof}
At a fixed pair $t,1-t$, smooth interpolation allows independent choices of the values of $v$ and $b$ at both points. The corresponding matrix is
\[
\begin{pmatrix}v(t)-b(t)/h(t)&b(t)/h(t)\\b(1-t)/h(1-t)&v(1-t)-b(1-t)/h(1-t)\end{pmatrix},
\]
which can be any matrix $\begin{psmallmatrix}a_{11}&a_{12}\\a_{21}&a_{22}\end{psmallmatrix}$. A central operator therefore has scalar matrices on every two-point fiber. Its off-diagonal entries vanish, so $b=0$ away from $c$ and hence everywhere. Equality of its diagonal entries gives $v=Rv$. Conversely symmetric multipliers commute with every $T_{w,e}$.
\end{proof}

Put $\mathscr D_\beta=\{M_p+M_qR:p,q\in A_\beta\}$ and $J_\beta=\overline{\mathscr D_\beta}^{\|\cdot\|}$. For $b\in A_{\beta,0}$ define
\[
\Omega_\beta(b)=\min_{L\in\mathbb C}\lim_{r\downarrow0}[b-Lh]_{\beta,[c-r,c+r]}.
\]
\begin{theorem}\label{thm:distance}
The minimum defining $\Omega_\beta$ is attained, and
\[
\frac{\Omega_\beta(b)}{\|h\|_\beta}\le
\operatorname{dist}(T_{v,b},\mathscr D_\beta)\le C_\beta\Omega_\beta(b).
\]
Consequently,
\[
J_\beta=\{T_{v,b}\in\mathscr C_\beta:\Omega_\beta(b)=0\}.
\]
\end{theorem}
\begin{proof}
For fixed $b$, set $\Phi_b(L)=\lim_{r\downarrow0}[b-Lh]_{\beta,[c-r,c+r]}$. This limit exists by monotonicity in $r$. The seminorm inequality makes $\Phi_b$ Lipschitz in $L$, while
\[
\Phi_b(L)\ge 2^{1-\beta}|L|- [b]_\beta,
\]
because $[h]_{\beta,[c-r,c+r]}=2^{1-\beta}$. Thus $\Phi_b$ is continuous and coercive and attains its minimum.

A regular operator $M_p+M_qR$ has the form $T_{p+q,qh}$. Since $(q-q(c))h$ has vanishing local $\beta$-H\"older seminorm at $c$, $\Omega_\beta(b)\le\|b-qh\|_\beta$. Apply $T_{v,b}-T_{p+q,qh}$ to $h$ and use the coefficient recovery estimate to obtain the lower bound.

For the upper bound, choose $L$ minimizing $\Phi_b$ and put $d=b-Lh$.
Then $d(c)=0$. Fix a Lipschitz function $\chi:\mathbb R\to[0,1]$
which equals $1$ on $[-1,1]$ and vanishes outside $[-2,2]$, and set
$\chi_r(t)=\chi((t-c)/r)$ and $d_r=(1-\chi_r)d$.
Write $K_r=[c-2r,c+2r]$ and $H_r=[d]_{\beta,K_r}$.
For $s,t\in K_r$, the product estimate and $|d(t)|\le H_r|t-c|^\beta$
give
\[
|\chi_r(t)d(t)-\chi_r(s)d(s)|
\le H_r|t-s|^\beta+C H_r r^{\beta-1}|t-s|.
\]
Since $|t-s|\le4r$ on $K_r$, the right-hand side is at most
$C_\beta H_r|t-s|^\beta$. If exactly one point lies outside $K_r$,
compare it with the intervening boundary point where $\chi_r d=0$;
this gives the same bound. Also $\|\chi_r d\|_\infty\le C r^\beta H_r$.
Consequently
\[
\|d-d_r\|_\beta=\|\chi_r d\|_\beta\le C_\beta H_r.
\]
The cutoff belongs to $A_\beta$ because it is Lipschitz. Hence
$d_r\in A_\beta$. Since $d_r$ vanishes near $c$, the function
$q_r=d_r/h$, extended by zero around $c$, belongs to $A_\beta$:
away from $c$ the reciprocal of $h$ is smooth, and multiplication by
smooth cutoff functions preserves $A_\beta$.
Thus $T_{v,Lh+d_r}=M_{v-L-q_r}+M_{L+q_r}R$ is regular.
The coefficient estimate in Theorem~\ref{thm:algebra} now yields
\[
\operatorname{dist}(T_{v,b},\mathscr D_\beta)
\le C_\beta\|b-(Lh+d_r)\|_\beta
\le C_\beta H_r.
\]
Letting $r\downarrow0$ gives the upper estimate because
$H_r\downarrow\Phi_b(L)=\Omega_\beta(b)$.
The lower and upper estimates together give the stated characterization
of $J_\beta$.
\end{proof}

The use of local H\"older seminorms in approximation problems has classical precedents in distance formulas for spaces defined by oscillation; see Perfekt \cite{Perfekt2013} and Angrisani et al. \cite{Angrisani2020}. Theorem~\ref{thm:distance} addresses the operator-norm closure of regular-coefficient operators inside $\mathscr C_\beta$. Its quantity $\Omega_\beta(b)$ measures the local obstruction after allowing a scalar correction $Lh$. We use the classical approximation viewpoint without claiming that the underlying cutoff or oscillation method is new; the precise relation of this operator-theoretic characterization to the abstract distance frameworks remains to be established.

\begin{example}\label{ex:strict}
Both inclusions $\mathscr D_\beta\subsetneq J_\beta\subsetneq\mathscr C_\beta$ are strict.
\end{example}
\begin{proof}
Choose $0<\alpha<\beta$. The function $b_1(t)=|t-c|^{\beta+\alpha}$ belongs to $A_{\beta,0}$, and $[b_1]_{\beta,[c-r,c+r]}\to0$. Hence $T_{0,b_1}\in J_\beta$. But $b_1/h=\operatorname{sgn}(t-c)|t-c|^\alpha$ is not $\beta$-H\"older, so this operator is not regular. For $b_2(t)=|t-c|^\beta$, the one-sided quotients of $b_2-Lh$ at $c$ give $\Omega_\beta(b_2)\ge\min_L\max(|1-L|,|1+L|)=1$. Thus $T_{0,b_2}\notin J_\beta$.
\end{proof}

\begin{lemma}\label{lem:bumps}
For $n\ge1$, let $x_n=2^{-(n+1)}$, $w_n=x_n^2$, and $I_n=[c+x_n,c+x_n+w_n]$. These intervals are pairwise disjoint and contained in $(c,1)$. Let $\varphi(s)=\max\{0,1-|2s-1|\}$ on $[0,1]$, extended by zero, and put
\[
b_\alpha(t)=\sum_n\alpha_n w_n^\beta\varphi\!\left(\frac{t-c-x_n}{w_n}\right),\qquad \alpha\in\ell^\infty.
\]
Then $b_\alpha\in A_{\beta,0}$ and $\|b_\alpha\|_\beta\le C_\beta\|\alpha\|_\infty$. If $\alpha\in c_0$, then $\lim_{r\downarrow0}[b_\alpha]_{\beta,[c-r,c+r]}=0$.
\end{lemma}
\begin{proof}
Write $f_n(t)=\alpha_n w_n^\beta\varphi((t-c-x_n)/w_n)$, extended by zero outside $I_n$, and put $A=\|\alpha\|_\infty$. Since $\varphi$ is $2$-Lipschitz, bounded by $1$, and vanishes at both endpoints, for all $s,t\in[0,1]$,
\[
|f_n(t)-f_n(s)|\le |\alpha_n|\min\{2w_n^{\beta-1}|t-s|,\,2w_n^\beta\}
\le 2|\alpha_n||t-s|^\beta.
\]
In particular, each zero-extended bump has $\beta$-H\"older seminorm at most $2|\alpha_n|$. At any point, at most one bump is nonzero. If $s,t$ belong to the same support, the preceding estimate applies. If they belong to distinct supports (or one is outside every support), choose $z_s$ on the boundary of the support containing $s$, between $s$ and $t$, and similarly choose $z_t$ for $t$ when needed. Because the bumps vanish at their support boundaries,
\[
|b_\alpha(s)-b_\alpha(t)|
\le 2A\bigl(|s-z_s|^\beta+|t-z_t|^\beta\bigr)
\le 4A|s-t|^\beta.
\]
(The term for a point outside the supports is zero.) Hence $[b_\alpha]_\beta\le4A$, while $\|b_\alpha\|_\infty\le A\sup_n w_n^\beta$. The sum defines a continuous function at $c$ since $w_n^\beta\to0$, and $b_\alpha(c)=0$.

Each triangular bump has total variation $2|\alpha_n|w_n^\beta$. Thus
\[
\operatorname{Var}_{[0,1]}\Bigl(\sum_{n\ge N}f_n\Bigr)
\le 2A\sum_{n\ge N}w_n^\beta\longrightarrow0.
\]
Every finite sum belongs to $AC([0,1])\cap C^\beta([0,1])$ and therefore to $A_\beta$ by Lemma~\ref{lem:ac}. For a family of disjoint intervals occurring in the definition of restricted variation, the sum of absolute increments of the tail is bounded by its total variation, uniformly in the family. Splitting $b_\alpha$ into a finite initial sum and a tail, and then letting first the total interval length tend to zero and next $N\to\infty$, proves $b_\alpha\in A_{\beta,0}$.

Finally, if $\alpha\in c_0$, then for every $N$ there exists $r_N>0$ such that $[c-r_N,c+r_N]$ meets only supports with indices $n\ge N$. The same two-point estimates, applied to this tail, give
\[
[b_\alpha]_{\beta,[c-r_N,c+r_N]}
\le4\sup_{n\ge N}|\alpha_n|\longrightarrow0.
\]
Monotonicity of the local seminorm as $r\downarrow0$ completes the proof.
\end{proof}

\begin{theorem}\label{thm:quotient}
The Banach space quotient $\mathscr C_\beta/J_\beta$ contains a complemented subspace isomorphic to $\ell^\infty/c_0$.
\end{theorem}
\begin{proof}
Define $F:\ell^\infty\to\mathscr C_\beta$ by $F(\alpha)=T_{0,b_\alpha}$ and define
\[
S(T_{v,b})_n=\frac{b(c+x_n+w_n/2)-b(c+x_n)}{(w_n/2)^\beta}.
\]
The coefficient estimate in Theorem~\ref{thm:algebra} makes $S$ bounded,
since each displayed quotient is bounded by $[b]_\beta$.
Lemma~\ref{lem:bumps} gives $SF=2^\beta I$.

We verify that $S(J_\beta)\subset c_0$. If
$T_{v,b}\in J_\beta$, Theorem~\ref{thm:distance} gives
$\Omega_\beta(b)=0$. Choose $L\in\mathbb C$ attaining the minimum;
then
\[
\lim_{r\downarrow0}[b-Lh]_{\beta,[c-r,c+r]}=0.
\]
Put $s_n=c+x_n$ and $t_n=s_n+w_n/2$. The contribution of
$b-Lh$ to $S(T_{v,b})_n$ tends to zero, because
$s_n,t_n\to c$ and $|t_n-s_n|=w_n/2$.
For the remaining term, the mean value theorem on the right of $c$
gives
\[
\frac{|h(t_n)-h(s_n)|}{(w_n/2)^\beta}
\le \beta x_n^{\beta-1}(w_n/2)^{1-\beta}
=\beta 2^{\beta-1}x_n^{1-\beta}\longrightarrow0,
\]
where $w_n=x_n^2$ and $0<\beta<1$. Hence
$S(T_{v,b})\in c_0$.
Conversely, if $\alpha\in c_0$, Lemma~\ref{lem:bumps}
shows $\Omega_\beta(b_\alpha)=0$, so $F(c_0)\subset J_\beta$.
The bounded maps $F$ and $S$ therefore induce bounded maps
\[
\widetilde F:\ell^\infty/c_0\to\mathscr C_\beta/J_\beta,
\qquad
\widetilde S:\mathscr C_\beta/J_\beta\to\ell^\infty/c_0.
\]
Their composition is $\widetilde S\widetilde F=2^\beta I$.
Thus $\widetilde F$ is an isomorphic embedding and
$2^{-\beta}\widetilde F\widetilde S$ is a bounded projection
onto its range.
\end{proof}

Complemented copies of $\ell^\infty/c_0$ in Banach-space quotients have antecedents in the work of Plichko and W\'ojtowicz~\cite{PlichkoWojtowicz2007}. The preceding construction gives explicit maps in the present operator setting. Although $J_\beta$ is a closed subalgebra of $\mathscr C_\beta$, the quotient is considered only as a Banach space. In particular, the
complementation in Theorem~\ref{thm:quotient} is an assertion about the
Banach-space quotient: the maps $\widetilde F$ and $\widetilde S$
produce an explicit bounded projection, not an algebra homomorphism.

\begin{proposition}\label{prop:projections}
A bounded projection $P\in\mathscr C_\beta$ has kernel $E$ if and only if $P=P_a=M_a(I-R)$, where $ah\in A_{\beta,0}$ and
\[
a(t)+a(1-t)=1\qquad(t\ne c).
\]
\end{proposition}
\begin{proof}
Suppose $P=T_{v,b}$ is a projection with kernel $E$. Since $1\in E$, $P1=v=0$. Set $a=-b/h$, so $P=M_a(I-R)$ and $ah\in A_{\beta,0}$. Idempotence gives $a(t)(a(t)+a(1-t)-1)=0$ and its reflected counterpart. It remains to exclude simultaneous zeros of $a$ at a fiber.

The restriction $P|_O:O\to\operatorname{ran}P$ is a bounded bijection: injectivity follows from $O\cap E=\{0\}$ and surjectivity from $P(f)=P(P_Of)$. Since both spaces are Banach, the bounded inverse theorem gives $\|Pg\|_\beta\ge\eta\|g\|_\beta$ for $g\in O$ and some $\eta>0$. If $a(t_0)=a(1-t_0)=0$ at $t_0\ne c$, take odd triangular bumps of height $r^\beta$ supported in small neighborhoods of $t_0$ and $1-t_0$. Their H\"older norms are bounded below. More explicitly, on the two supports $a$ is $\beta$-H\"older and vanishes at their centers. Since $\|g_r\|_\infty=O(r^\beta)$ and $[g_r]_\beta=O(1)$, the product estimate gives
\[
[ag_r]_\beta\le \|a\|_{\infty,U_r}[g_r]_\beta+[a]_{\beta,U_r}\|g_r\|_\infty=O(r^\beta),
\]
where $U_r$ is the union of the two supports. The same estimate applies across their boundaries because $g_r$ vanishes there, and the supports stay a fixed positive distance apart. Thus $\|Pg_r\|_\beta=2\|ag_r\|_\beta\to0$, a contradiction. Thus $a$ and $Ra$ never vanish simultaneously. The two idempotence identities imply $a+Ra=1$ on every fiber.

Conversely, if $ah\in A_{\beta,0}$, Theorem~\ref{thm:multiplier} makes $P_a$ bounded. The identity $a+Ra=1$ gives $P_a^2=P_a$; clearly $E\subset\ker P_a$, and if $P_af=0$ then the same identity at each fiber forces $f=Rf$, so $\ker P_a=E$.
\end{proof}

\begin{proposition}\label{prop:projection-pairs}
If $P_a$ and $P_d$ are projections as in Proposition~\ref{prop:projections}, then
\[
P_aP_d=P_a,\qquad P_dP_a=P_d,\qquad(P_a-P_d)^2=0.
\]
In particular, $\sigma(P_a-P_d)=\{0\}$.
\end{proposition}
\begin{proof}
For every $f$, $(I-R)P_df=(d+Rd)(I-R)f=(I-R)f$. Hence $P_aP_d=P_a$, and the reverse identity follows similarly. Expanding the square proves nilpotence. Its spectrum is $\{0\}$.
\end{proof}

\begin{example}\label{ex:jump-projection}
There is a projection $P\in\mathscr C_\beta$ with $\ker P=E$ whose reflection coefficient has a jump at $c$, and $P\notin J_\beta$.
\end{example}
\begin{proof}
Let $a(t)=1$ for $t<c$ and $a(t)=0$ for $t>c$. Then $a+Ra=1$ and
\[
b=-ah=\begin{cases}|t-c|^\beta,&t<c,\\0,&t\ge c.\end{cases}
\]
Since $b\in AC\cap C^\beta$, Proposition~\ref{prop:projections} gives $P=T_{0,b}=M_a(I-R)$ with kernel $E$. For any $L\in\mathbb C$ and $r>0$,
\[
[b-Lh]_{\beta,[c-r,c+r]}\ge\max(|1+L|,|L|)\ge\tfrac12.
\]
Thus $\Omega_\beta(b)\ge1/2$, and Theorem~\ref{thm:distance} yields $P\notin J_\beta$.
\end{proof}

\begin{example}\label{ex:oscillatory-projection}
There is a projection $P\in\mathscr C_\beta$ with $\ker P=E$ whose reflection coefficient oscillates infinitely often near $c$, has no limit at $c$, and satisfies $P\notin J_\beta$.
\end{example}
\begin{proof}
For $t\ne c$, let
\[
a(t)=\frac12+\operatorname{sgn}(t-c)\sin\!\left(\log\frac1{|t-c|}\right).
\]
Then $a+Ra=1$. Put $b=-ah$ away from $c$ and $b(c)=0$. Thus
\[
b(t)=-\frac12h(t)-|t-c|^\beta\sin\!\left(\log\frac1{|t-c|}\right).
\]
For $G(x)=x^\beta\sin(\log(1/x))$, $G(0)=0$, one has
\[
|G'(x)|\le(\beta+1)x^{\beta-1},\qquad
|G(y)-G(x)|\le\frac{\beta+1}{\beta}|y-x|^\beta.
\]
Hence $G$ is absolutely continuous and $\beta$-H\"older; so is $G(|t-c|)$. Lemma~\ref{lem:ac} implies $b\in A_{\beta,0}$. Proposition~\ref{prop:projections} supplies the projection. The sine term shows that $a$ has no limit at $c$. For any $L\in\mathbb C$,
\[
[b-Lh]_{\beta,[c-r,c+r]}\ge\sup_{0<x\le r}|-\tfrac12-L-\sin(\log(1/x))|
\ge\max(|L+\tfrac32|,|L-\tfrac12|)\ge1.
\]
Thus $\Omega_\beta(b)\ge1$ and $P\notin J_\beta$ by Theorem~\ref{thm:distance}.
\end{proof}

\section{Invertibility and spectrum}
For $T=M_p+M_qR$, $p,q\in A_\beta$, put $\Delta_T=p(Rp)-q(Rq)\in E$.
\begin{theorem}\label{thm:invertibility}
The operator $T=M_p+M_qR$ is invertible on $A_\beta$ if and only if $\Delta_T$ has no zeros on $[0,1]$. In that case,
\[
T^{-1}=M_{\frac{Rp}{\Delta_T}}-M_{\frac q{\Delta_T}}R.
\]
\end{theorem}
\begin{proof}
If $\Delta_T$ has no zeros, its reciprocal belongs to $A_\beta$ by the Banach-algebra inverse-closedness under smooth composition. Direct multiplication gives
\[
T(M_{Rp}-M_qR)=(M_{Rp}-M_qR)T=M_{\Delta_T}.
\]
Since $\Delta_T$ is symmetric, the displayed inverse formula follows.

Conversely, for $t<c$ let $E_tf=(f(t),f(1-t))^\top$. This map is surjective and intertwines $T$ with
\[
A_t(T)=\begin{pmatrix}p(t)&q(t)\\q(1-t)&p(1-t)\end{pmatrix}.
\]
If $T$ is onto, the matrix $A_t(T)$ is onto, hence invertible, so $\Delta_T(t)\ne0$. To treat $c$, note that every eigenvalue of $A_t(T)$ belongs to $\sigma(T)$: otherwise the invertible operator $T-\lambda I$ would induce a surjective $A_t(T)-\lambda I$. If $\Delta_T(c)=0$, continuity of the characteristic polynomials yields eigenvalues $\lambda_n\in\sigma(A_{t_n}(T))$ with $t_n\uparrow c$ and $\lambda_n\to0$. Closedness of $\sigma(T)$ would imply $0\in\sigma(T)$, a contradiction. Symmetry of $\Delta_T$ completes the proof.
\end{proof}

\begin{theorem}\label{thm:spectrum}
For $T=M_p+M_qR$ with $p,q\in A_\beta$,
\[
\sigma(T)=\bigcup_{0\le t\le c}\sigma\!\left(\begin{pmatrix}p(t)&q(t)\\q(1-t)&p(1-t)\end{pmatrix}\right).
\]
\end{theorem}
\begin{proof}
For every $\lambda\in\mathbb C$,
\[
\Delta_{T-\lambda I}(t)=(p(t)-\lambda)(p(1-t)-\lambda)-q(t)q(1-t)
=\det(A_t(T)-\lambda I).
\]
Apply Theorem~\ref{thm:invertibility} to $T-\lambda I$.
\end{proof}

\begin{corollary}\label{cor:symmetric-spectrum}
Let $p,q\in E$ and $T=M_p+M_qR$. Then
\[
\sigma(T)=(p+q)([0,1])\cup(p-q)([0,1]).
\]
Moreover, $T$ is invertible if and only if both $p+q$ and $p-q$ have no zeros on $[0,1]$.
\end{corollary}
\begin{proof}
Since $Rp=p$ and $Rq=q$, the matrix
\[
A_t(T)=\begin{pmatrix}p(t)&q(t)\\q(t)&p(t)\end{pmatrix}
\]
has eigenvalues $p(t)+q(t)$ and $p(t)-q(t)$. These functions are symmetric, so their ranges over $[0,1/2]$ equal their ranges over $[0,1]$. The spectral formula follows from Theorem~\ref{thm:spectrum}. Also, $\Delta_T=p^2-q^2=(p+q)(p-q)$, and the invertibility assertion follows from Theorem~\ref{thm:invertibility}.
\end{proof}

\begin{example}\label{ex:spectral-gap}
For $\kappa\ge0$, consider $T_\kappa=M_t+\kappa R$, where $M_t$ denotes multiplication by the identity function. Then
\[
\sigma(T_\kappa)=
\left[\frac12-\sqrt{\kappa^2+\frac14},\frac12-\kappa\right]
\cup
\left[\frac12+\kappa,\frac12+\sqrt{\kappa^2+\frac14}\right].
\]
In particular, $T_\kappa$ is invertible if and only if $\kappa>1/2$.
\end{example}
\begin{proof}
For $0\le t\le1/2$, the associated matrix is
\[
A_t(T_\kappa)=\begin{pmatrix}t&\kappa\\\kappa&1-t\end{pmatrix},
\]
with eigenvalues
\[
\lambda_\pm(t)=\frac12\pm\sqrt{(t-\frac12)^2+\kappa^2}.
\]
As $t$ ranges over $[0,1/2]$, the square root ranges over $[\kappa,\sqrt{\kappa^2+1/4}]$. Theorem~\ref{thm:spectrum} yields the asserted spectrum. Finally,
\[
\Delta_{T_\kappa}(t)=t(1-t)-\kappa^2.
\]
Since $t(1-t)$ ranges over $[0,1/4]$, Theorem~\ref{thm:invertibility} shows that $T_\kappa$ is invertible exactly when $\kappa>1/2$.
\end{proof}

\begin{lemma}\label{lem:weighted-quotient}
Let $b,g,d\in A_\beta$ vanish at $c$, and suppose that, for some
$B,G,\delta_0>0$ and all $t\ne c$,
\[
 |b(t)|\le B|h(t)|,\qquad |g(t)|\le G|h(t)|,
 \qquad |d(t)|\ge\delta_0|h(t)|.
\]
Then the function defined by $F(t)=b(t)g(t)/d(t)$ for $t\ne c$ and
$F(c)=0$ belongs to $A_\beta$.
\end{lemma}
\begin{proof}
The bounds give $|F(t)|\le BG|h(t)|/\delta_0$, so $F$ is continuous
at $c$. For $s,t\ne c$, interchange the points if necessary so that
$|h(s)|\le|h(t)|$. The identity
\[
\begin{aligned}
 F(t)-F(s)={}&\frac{g(t)}{d(t)}(b(t)-b(s))
 +\frac{b(s)}{d(t)}(g(t)-g(s))\\
 &+\frac{b(s)g(s)}{d(t)d(s)}(d(s)-d(t))
\end{aligned}
\]
and the hypotheses imply
\begin{equation}\label{eq:weighted-quotient-increments}
 |F(t)-F(s)|\le\frac G{\delta_0}|b(t)-b(s)|
 +\frac B{\delta_0}|g(t)-g(s)|
 +\frac{BG}{\delta_0^2}|d(t)-d(s)|.
\end{equation}
The same bound holds when one point is $c$, by continuity. It gives
the $\beta$-H\"older estimate. Summing
\eqref{eq:weighted-quotient-increments} over each family of intervals
admissible in the definition of $V_{\beta,\varepsilon,M}$ shows that
\[
\begin{aligned}
 V_{\beta,\varepsilon,M}(F)\le{}&
 \frac G{\delta_0}V_{\beta,\varepsilon,M}(b)
 +\frac B{\delta_0}V_{\beta,\varepsilon,M}(g)\\
 &+\frac{BG}{\delta_0^2}V_{\beta,\varepsilon,M}(d)
&+\frac{BG}{\delta_0^2}V_{\beta,\varepsilon,M}(d)
 \longrightarrow0\qquad(\varepsilon\downarrow0)
\end{aligned}
\]
for every $M>0$. Therefore $F\in A_\beta$.
\end{proof}

\begin{lemma}\label{lem:localized-spectral-test}
Let $J\subset\mathbb R$ be a compact interval, let $y_0$ be an
interior point, and let $H$ be an $m\times m$ matrix field whose
entries are $\beta$-H\"older on a neighborhood of $y_0$.
Suppose $H(y_0)\xi=0$ for some $\xi\in\mathbb C^m\setminus\{0\}$.
For a nonzero $\eta\in C_c^\infty((-1,1))$ and sufficiently small
$r>0$, put
\[
 F_r(y)=r^\beta\eta\!\left(\frac{y-y_0}{r}\right)\xi.
\]
Then there are constants $c,C>0$, independent of $r$, such that
\[
 c\le\|F_r\|_{C^\beta(J;\mathbb C^m)}\le C,\qquad
 \|F_r\|_\infty\le Cr^\beta,\qquad
 \|HF_r\|_{C^\beta(J;\mathbb C^m)}\le Cr^\beta,
\]
where $HF_r$ is extended by zero away from the support of $F_r$.
All the components of $F_r$ belong to $A_\beta(J)$.
\end{lemma}
\begin{proof}
On $|y-y_0|\le r$, let $q(y)=H(y)\xi$. Since $q(y_0)=0$,
\[
 \|q\|_{\infty,[y_0-r,y_0+r]}\le C_1r^\beta\|\xi\|,
 \qquad [q]_{\beta,[y_0-r,y_0+r]}\le C_1\|\xi\|.
\]
The $\beta$-H\"older seminorm of
$r^\beta\eta((\cdot-y_0)/r)$ is $[\eta]_\beta$,
independent of $r$, and its supremum norm is at most
$r^\beta\|\eta\|_\infty$. The H\"older product estimate gives
\[
 \|HF_r\|_\infty\le C_2r^{2\beta},\qquad
 [HF_r]_\beta\le C_3r^\beta.
\]
The extension by zero obeys the same bounds because $\eta$ is
compactly supported inside $(-1,1)$. Direct scaling gives
$[F_r]_\beta=[\eta]_\beta\|\xi\|>0$, and also its uniform
upper bound. The functions $F_r$ are smooth and compactly supported
in the interior of $J$, so they belong to $A_\beta(J)$.
\end{proof}

For $T=T_{v,b}\in\mathscr C_\beta$ put $a=b/h$ on
$[0,1]\setminus\{c\}$. For $0\le t<c$ define its two-point matrix
\begin{equation}\label{eq:singular-fiber-matrix}
 B_t(T)=
 \begin{pmatrix}
 v(t)-a(t)&a(t)\\
 a(1-t)&v(1-t)-a(1-t)
 \end{pmatrix}.
\end{equation}
Unlike the regular matrices of Theorem~\ref{thm:spectrum}, these
matrices need not have a limit as $t\uparrow c$.

\begin{theorem}\label{thm:full-essential-spectrum}
For every $T=T_{v,b}\in\mathscr C_\beta$,
\begin{equation}\label{eq:full-essential-spectrum}
 \boxed{\quad
 \sigma_{\mathrm{ap}}(T)=\sigma_{\mathrm{usf}}(T)
 =\sigma_{\mathrm{ess}}(T)=\sigma(T)
 =\overline{\bigcup_{0\le t<c}\sigma(B_t(T))}.\quad}
\end{equation}
Moreover, $\mathscr C_\beta$ is inverse-closed in
$\mathcal B(A_\beta)$: whenever $T$ is invertible on $A_\beta$,
its inverse also belongs to $\mathscr C_\beta$.
\end{theorem}
\begin{proof}
Set $\Sigma=\overline{\bigcup_{0\le t<c}\sigma(B_t(T))}$.
First let $t_0\in(0,c)$ and
$B_{t_0}(T)\xi=\lambda\xi$ for $\xi=(\xi_1,\xi_2)^\mathsf T\ne0$.
Choose a nonzero, even $\eta\in C_c^\infty((-1,1))$ and, for small
$r>0$, set
\[
 f_r(t)=r^\beta\left[
 \xi_1\eta\!\left(\frac{t-t_0}{r}\right)
 +\xi_2\eta\!\left(\frac{t-(1-t_0)}{r}\right)\right].
\]
The supports are disjoint and avoid $c$ and the endpoints. These
functions are smooth, hence belong to $A_\beta$. Their $\beta$-H\"older
norms are bounded above and away from zero, while
$\|f_r\|_\infty=O(r^\beta)$. The quantitative product bound
in Lemma~\ref{lem:localized-spectral-test} applies in the paired
coordinates $(f_r(t),f_r(1-t))$. Since the entries of $B_t(T)$ are
locally $\beta$-H\"older near $t_0$, the eigenvector relation and
the H\"older product estimate on the two supports give
\[
 \|(T-\lambda I)f_r\|_\beta=O(r^\beta).
\]
After normalizing and taking $r\downarrow0$, we obtain a sequence
of unit vectors $u_r$ such that $\|(T-\lambda I)u_r\|_\beta\to0$
and $\|u_r\|_\infty\to0$. In particular, $(u_r)$ has no
norm-convergent subsequence in $A_\beta$: a norm limit would have norm one and vanish identically. If $T-\lambda I$ had
closed range and finite-dimensional kernel, the open mapping theorem applied to the induced bijection $A_\beta/\ker(T-\lambda I)\to\operatorname{ran}(T-\lambda I)$
would give
\[
 \operatorname{dist}(u_r,\ker(T-\lambda I))\le C\|(T-\lambda I)u_r\|_\beta\to0.
\]
The boundedness of $(u_r)$ and finite-dimensionality of the kernel
would then give a norm-convergent subsequence, a contradiction.
Thus $\lambda\in\sigma_{\mathrm{usf}}(T)\cap\sigma_{\mathrm{ap}}(T)$.
Both of these spectral sets are closed, so passage to $t_0=0$
and to all limit points proves
$\Sigma\subseteq\sigma_{\mathrm{usf}}(T)\cap\sigma_{\mathrm{ap}}(T)$.

Conversely, let $\lambda\notin\Sigma$ and put
$v_\lambda=v-\lambda$. The matrices $B_t(T)-\lambda I$ have
uniformly bounded entries, since $b(c)=0$ and $b\in C^\beta$
imply $|b/h|\le[b]_\beta$. Since their eigenvalues are uniformly
separated from zero, their determinants are uniformly bounded
away from zero. Define
\[
 d_\lambda=h\,v_\lambda Rv_\lambda
 +v_\lambda Rb-bRv_\lambda\in A_{\beta,0}.
\]
Direct expansion gives, for $t<c$ (and by reflection for $t>c$),
\[
 \det(B_t(T)-\lambda I)=\frac{d_\lambda(t)}{h(t)}.
\]
Consequently, for some $\delta_\lambda>0$,
$|d_\lambda(t)|\ge\delta_\lambda|h(t)|$ for every $t\ne c$.
We also have $v_\lambda(c)\ne0$: otherwise the H\"older bounds
for $v_\lambda,b$ at $c$ would imply
$|d_\lambda(t)|=O(|t-c|^{2\beta})=o(|h(t)|)$.

For $t\ne c$ define
\[
 e_\lambda(t)=-\frac{b(t)h(t)}{d_\lambda(t)},
 \qquad
 w_\lambda(t)=\frac{h(t)Rv_\lambda(t)+Rb(t)-b(t)}{d_\lambda(t)},
\]
with $e_\lambda(c)=0$ and $w_\lambda(c)=1/v_\lambda(c)$.
Lemma~\ref{lem:weighted-quotient}, applied to $b,h,d_\lambda$,
gives $e_\lambda\in A_{\beta,0}$. To verify
$w_\lambda\in A_\beta$, choose an interval about $c$ on which
$v_\lambda$ is nonvanishing and use
\begin{equation}\label{eq:inverse-symbol-decomposition}
 w_\lambda=\frac1{v_\lambda}
 +\frac{b(Rv_\lambda-v_\lambda)}{v_\lambda d_\lambda}.
\end{equation}
The function $Rv_\lambda-v_\lambda$ lies in $A_{\beta,0}$ and
is bounded in modulus by a constant times $|h|$. Applying
Lemma~\ref{lem:weighted-quotient} to
$b,Rv_\lambda-v_\lambda,d_\lambda$ shows that the second
numerator quotient belongs to $A_\beta$; multiplication by the
local $A_\beta$ reciprocal of $v_\lambda$ preserves this property.
Away from $c$ the denominator $d_\lambda$ does not vanish, so
Lemma~\ref{lem:composition-reciprocal} applies. Smooth cutoffs
combine the local formulas to yield $w_\lambda\in A_\beta$.

The inverse of $B_t(T)-\lambda I$ has off-diagonal entries
$e_\lambda(t)/h(t)$ and $e_\lambda(1-t)/h(1-t)$, and its row sums
are $w_\lambda(t)$ and $w_\lambda(1-t)$. Hence the two-point
matrix of $T_{w_\lambda,e_\lambda}$ equals
$(B_t(T)-\lambda I)^{-1}$ for $t<c$. At $c$,
$T-\lambda I$ acts as multiplication by $v_\lambda(c)$, and
$T_{w_\lambda,e_\lambda}$ acts as multiplication by its reciprocal.
Thus the two operators are mutual inverses on $A_\beta$.
We have proved $\sigma(T)\subseteq\Sigma$. The standard inclusions
$\sigma_{\mathrm{usf}}(T)\subseteq\sigma_{\mathrm{ess}}(T)
\subseteq\sigma(T)$, together with
$\sigma_{\mathrm{ap}}(T)\subseteq\sigma(T)$, yield all the
claimed equalities. Applying the inverse construction with
$\lambda=0$ proves inverse-closedness.
\end{proof}

\begin{corollary}\label{cor:full-fredholm}
For $T\in\mathscr C_\beta$, the following conditions are equivalent:
$T$ is Fredholm; $T$ is upper semi-Fredholm; $T$ is invertible;
and $\det B_t(T)\ne0$ uniformly in the sense that
$\inf_{0\le t<c}|\det B_t(T)|>0$.
\end{corollary}
\begin{proof}
The spectral equalities in Theorem~\ref{thm:full-essential-spectrum}
show that all three operator properties are equivalent.
The determinant condition is equivalent to
$0\notin\overline{\bigcup_{t<c}\sigma(B_t(T))}$:
indeed, the matrices $B_t(T)$ are uniformly bounded, and for
$2\times2$ matrices uniform separation of the determinants from
zero is equivalent to uniform separation of their eigenvalues from
zero. Apply \eqref{eq:full-essential-spectrum}.
\end{proof}

\begin{corollary}
\label{cor:full-symmetric-fredholm}
Let $u(t)=t(1-t)$ and $T\in\{M_u\}'$. By
Theorem~\ref{cor:full-symmetric} there are unique
$v\in A_\beta$ and $b\in A_{\beta,0}$ with $T=T_{v,b}$.
Let $B_t(T)$ be the matrix in \eqref{eq:singular-fiber-matrix}.
Then
\[
\sigma_{\mathrm{ap}}(T)=\sigma_{\mathrm{usf}}(T)
=\sigma_{\mathrm{ess}}(T)=\sigma(T)
=\overline{\bigcup_{0\le t<c}\sigma(B_t(T))}.
\]
Moreover, $T$ is Fredholm if and only if it is invertible,
and this is equivalent to
$\inf_{0\le t<c}|\det B_t(T)|>0$.
\end{corollary}
\begin{proof}
Apply Theorem~\ref{thm:full-essential-spectrum} and
Corollary~\ref{cor:full-fredholm} using the equality
$\{M_u\}'=\mathscr C_\beta$ from
Theorem~\ref{cor:full-symmetric}.
\end{proof}

\begin{example}\label{ex:essential-oscillatory}
Let $P$ be the projection in Example~\ref{ex:oscillatory-projection}
and put $S_\theta=I+\theta P$ for $\theta\in\mathbb C\setminus\{0\}$.
Since $P^2=P$, $P\ne0$, and $\ker P=E\ne\{0\}$,
\[
 \sigma_{\mathrm{ess}}(S_\theta)=\sigma(S_\theta)
 =\{1,1+\theta\}.
\]
For $\theta\ne-1$,
\[
 S_\theta^{-1}=I-\frac{\theta}{1+\theta}P.
\]
When $\theta\notin\{0,-1\}$, both $S_\theta$ and
$S_\theta^{-1}$ lie in $\mathscr C_\beta\setminus J_\beta$:
otherwise, subtraction of the identity and division by a
nonzero scalar would put $P$ in $J_\beta$, contrary to
Example~\ref{ex:oscillatory-projection}.
\end{example}

\section{Finite-branch multipliers: representation, spectrum, and local closure}
\label{sec:finite-branches}

Theorem~\ref{thm:exact-finite-fiber} already gives locality
for every function in $A_\beta$, so no second locality proof is needed.
Here we specialize its formula to finite full branches and investigate
the regular finite-branch operator algebra. The branch assumption is
essential for the composition operators and matrix representations;
finite fibers alone do not guarantee bounded inverse-branch composition.

For an integer $n\ge2$, partition $[0,1]$ into
$I_k=[(k-1)/n,k/n]$, $1\le k\le n$, and define the continuous
piecewise affine \emph{full-branch tent map}
\[
u_n(t)=\begin{cases}
nt-(k-1),&t\in I_k,\ k\text{ odd},\\
k-nt,&t\in I_k,\ k\text{ even}.
\end{cases}
\]
Each branch maps onto $[0,1]$. Write
$\tau_k:[0,1]\to I_k$ for its affine inverse, set
$\phi_k=\tau_k\circ u_n$, and let $C_{\phi_k}f=f\circ\phi_k$.
For $0<y<1$, the $n$ points $\tau_k(y)$ are distinct and comprise
the level set $u_n^{-1}(y)$. At $y=0$ or $y=1$, some of these points
coalesce; the matrices below retain the limiting $n$-sheet structure.

\begin{lemma}\label{lem:piecewise-composition}
If $\phi:[0,1]\to[0,1]$ is continuous and piecewise affine with
finitely many pieces, then $C_\phi$ is a bounded operator on $A_\beta$.
In particular each $C_{\phi_k}$ is bounded, and
$C_{\phi_j}C_{\phi_k}=C_{\phi_k}$.
\end{lemma}
\begin{proof}
Let $L$ be a Lipschitz constant for $\phi$ and let $N$ be the number
of affine pieces. The ordinary H\"older bound is
$[f\circ\phi]_\beta\le L^\beta[f]_\beta$.
Given disjoint-interior intervals $[a_j,b_j]$ in the definition of
$V_{\beta,\delta,M}$, split each interval at the breakpoints of
$\phi$. The total lengths remain at most $\delta$, while concavity
of $x^\beta$ gives
\[
\sum_{j,k}|I_{j,k}|^\beta\le N^{1-\beta}M.
\]
Within any fixed affine piece of $\phi$, the images of the split
intervals have disjoint interiors (after discarding zero-length
images), total length at most $L\delta$, and sum of their
$\beta$-powers at most $L^\beta N^{1-\beta}M$. Therefore
\[
V_{\beta,\delta,M}(f\circ\phi)
\le N V_{\beta,L\delta,L^\beta N^{1-\beta}M}(f)
\longrightarrow0\qquad(\delta\downarrow0).
\]
Thus $f\circ\phi\in A_\beta$, and the H\"older bound gives
boundedness. Finally,
$u_n\circ\phi_j=u_n$, so
$\phi_k\circ\phi_j=\tau_k\circ u_n\circ\phi_j=\phi_k$.
The asserted identity follows by evaluating the composition
operators on functions.
\end{proof}

\begin{corollary}\label{thm:finite-locality}
Let $S\in\mathcal B(A_\beta)$ commute with $M_{u_n}$ and
$0<y<1$. For each $t=\tau_i(y)$ there are unique scalars
$a_{i1}(y),\ldots,a_{in}(y)$ such that
\[
 (Sf)(t)=\sum_{j=1}^n a_{ij}(y)f(\tau_j(y))
 \qquad\text{for every }f\in A_\beta.
\]
\end{corollary}
\begin{proof}
The level set $u_n^{-1}(\{y\})$ consists of $n$ distinct
points. Apply Theorem~\ref{thm:exact-finite-fiber}; smooth
interpolation gives uniqueness.
\end{proof}

Put
\[
\mathscr R_{\beta,n}
=\left\{M_v+\sum_{j=1}^n M_{p_j}C_{\phi_j}:
          v,p_1,\ldots,p_n\in A_\beta\right\}.
\]
This is the \emph{regular finite-branch class}. Its coefficients need not be unique: on $I_j$ one has $\phi_j(t)=t$,
which can make multiplication terms redundant. Nevertheless, the
operators and their fiber matrices are unambiguous.
The spectral theorem below concerns this regular algebra; the full commutant is treated in Section~\ref{sec:full-commutant}.
For $n=2$, one has $C_{\phi_1}+C_{\phi_2}=I+R$. Thus
$M_p+M_qR=M_{p-q}+M_q(C_{\phi_1}+C_{\phi_2})$;
in particular the original regular multiplication--reflection
operators are included in $\mathscr R_{\beta,2}$. The latter also
allows independent branch coefficients, so it is a larger class.

\begin{theorem}
\label{thm:finite-spectrum}
The class $\mathscr R_{\beta,n}$ is a unital algebra contained in
the commutant of $M_{u_n}$. For
$T=M_v+\sum_{j=1}^nM_{p_j}C_{\phi_j}$ define
\[
A_y(T)=\left(v(\tau_i(y))\delta_{ij}
                  +p_j(\tau_i(y))\right)_{1\le i,j\le n},
\qquad 0\le y\le1.
\]
Then
\[
\boxed{\ \sigma_{\mathcal B(A_\beta)}(T)
=\bigcup_{0\le y\le1}\sigma(A_y(T)).\ }
\]
In particular, $T$ is invertible on $A_\beta$ if and only if
$\det A_y(T)\ne0$ for every $y\in[0,1]$; whenever this holds,
$T^{-1}\in\mathscr R_{\beta,n}$.
\end{theorem}
\begin{proof}
The maps $\phi_j$ are continuous and piecewise affine, so their
composition operators are bounded by Lemma~\ref{lem:piecewise-composition}.
Since $u_n\circ\phi_j=u_n$, every member of the displayed class
commutes with $M_{u_n}$. Its closure under multiplication follows
from
\[
C_{\phi_j}M_a=M_{a\circ\phi_j}C_{\phi_j},
\qquad C_{\phi_j}C_{\phi_k}=C_{\phi_k},
\]
and the fact that $A_\beta$ is a Banach algebra stable under these
compositions. The identity is $M_1$.

For $y\in[0,1]$ let
$E_yf=(f(\tau_1(y)),\ldots,f(\tau_n(y)))^\top$. Then
$E_yT=A_y(T)E_y$, even if some coordinates coincide at an endpoint.
For $0<y<1$, the coordinates are distinct and $E_y$ is onto
$\mathbb C^n$ by smooth interpolation. Thus every eigenvalue of
$A_y(T)$ belongs to $\sigma(T)$: if $T-\lambda I$ had a bounded
inverse, $A_y(T)-\lambda I$ would be onto. By continuity of the
matrix entries, the same spectral inclusion holds at $y=0,1$,
because spectra of finite matrices vary continuously as compact
multisets and $\sigma(T)$ is closed.

For the reverse inclusion fix $\lambda\in\mathbb C$ and put
$S=T-\lambda I$. Write
\[
\det(zI-A_y(S))=z^n+c_{n-1}(y)z^{n-1}+\cdots+c_0(y).
\]
Every function $y\mapsto v(\tau_i(y))$ and
$y\mapsto p_j(\tau_i(y))$ belongs to $A_\beta([0,1])$:
the $\tau_i$ are affine, and the interval-variation condition is
preserved by affine changes of variable. Hence the elementary
symmetric coefficients $c_k$ belong to $A_\beta([0,1])$.
Lemma~\ref{lem:piecewise-composition} gives
$c_k\circ u_n\in A_\beta([0,1])$; their multiplication operators
commute with $S$. The matrix Cayley--Hamilton identity on every
fiber implies the operator identity
\[
S^n+M_{c_{n-1}\circ u_n}S^{n-1}+
\cdots+M_{c_0\circ u_n}=0.
\]
If $\det A_y(S)\ne0$ for all $y$, then
$c_0(y)=(-1)^n\det A_y(S)$ is nowhere zero. Its reciprocal, composed
with $u_n$, belongs to $A_\beta$ (apply the reciprocal map to a
nonvanishing member of the algebra). We can therefore solve this
identity for $S^{-1}$ as a polynomial in $S$ with central
coefficients:
\[
S^{-1}=-M_{(c_0\circ u_n)^{-1}}
\left(S^{n-1}+M_{c_{n-1}\circ u_n}S^{n-2}
+\cdots+M_{c_1\circ u_n}\right).
\]
In particular $S^{-1}\in\mathscr R_{\beta,n}$. Applying this with
$\lambda$ outside the union of the matrix spectra completes the
proof.
\end{proof}

\begin{example}
\label{ex:branch-regularity}
Fix $a>1/\beta$, and define the Lipschitz multiplier
\[
w_a(t)=\begin{cases}
2t,&0\le t\le\frac12,\\
(2-2t)^a,&\frac12\le t\le1.
\end{cases}
\]
Every noncritical level set has exactly two points, and the
inverse of the descending branch is
$\theta_2(y)=1-\tfrac12y^{1/a}$. Yet the associated selector
$\psi_2=\theta_2\circ w_a$ satisfies
\[
\psi_2(t)=1-\tfrac12(2t)^{1/a}
\quad(0\le t\le\tfrac12).
\]
Since $1/a<\beta$, the function $\psi_2$ is not
$\beta$-H\"older at $0$. Taking $f(t)=t\in A_\beta$ shows that
$f\circ\psi_2\notin C^\beta$, so $C_{\psi_2}$ does not act on
$A_\beta$. Hence the finite-fiber hypothesis alone does not
support the branch-operator representation used above.
\end{example}

The more delicate point is norm approximation at branch collisions.
For the next result set $\xi=1/n$, choose
$0<\eta<1/(4n)$, and let $h_\xi(t)
=\operatorname{sgn}(t-\xi)|t-\xi|^\beta$, with $h_\xi(\xi)=0$.
For $b\in A_\beta$ supported in $[\xi,\xi+\eta]$ and satisfying
$b(\xi)=0$, define
\[
(H_bf)(t)=\begin{cases}
\dfrac{b(t)}{h_\xi(t)}
       \bigl(f(\phi_1(t))-f(t)\bigr),&t\ne\xi,\\
0,&t=\xi.
\end{cases}
\]
The coefficient is zero away from the support of $b$. The first
branch is fixed by $\phi_1$, while for
$\xi\le t\le\xi+\eta$ one has $\phi_1(t)=2\xi-t$.
Consequently $H_b$ commutes with $M_{u_n}$.

Define the \emph{one-sided fold defect}
\[
\Omega^+_{\beta,\xi}(b)
=\min_{L\in\mathbb C}\lim_{r\downarrow0}
\bigl[b-L(t-\xi)^\beta\bigr]_{\beta,[\xi,\xi+r]}.
\]

\begin{theorem}
\label{thm:fold-distance}
The operator $H_b$ is bounded on $A_\beta$. The minimum in
$\Omega^+_{\beta,\xi}(b)$ exists, and constants $c_{\beta,n,\eta}>0$
and $C_{\beta,n,\eta}>0$ (independent of $b$) satisfy
\[
c_{\beta,n,\eta}\Omega^+_{\beta,\xi}(b)
\le\operatorname{dist}\bigl(H_b,\mathscr R_{\beta,n}\bigr)
\le C_{\beta,n,\eta}\Omega^+_{\beta,\xi}(b).
\]
Thus $H_b\in\overline{\mathscr R_{\beta,n}}^{\|\cdot\|}$ if and
only if $\Omega^+_{\beta,\xi}(b)=0$.
Moreover, for every $n\ge2$ the regular finite-branch algebra
$\mathscr R_{\beta,n}$ is not norm closed; its operator-norm closure
does not contain all operators $H_b$.
\end{theorem}
\begin{proof}
The proof of Lemma~\ref{lem:multiplier-bound} uses only the
vanishing of both factors at the distinguished point; the oddness
assumption on $g$ is not used in its estimates. The same argument,
with $c$ replaced by $\xi$, gives
\[
\left\|\frac{bg}{h_\xi}\right\|_\beta
\le K_\beta\|b\|_\beta\|g\|_\beta
\quad\text{if }b(\xi)=g(\xi)=0,
\]
with the quotient extended by zero at $\xi$.
Apply this to $g=f\circ\phi_1-f$, which belongs to $A_\beta$ by
Lemma~\ref{lem:piecewise-composition} and vanishes at $\xi$.
This proves boundedness of $H_b$.

Set
$\Phi_b(L)=\lim_{r\downarrow0}
[b-L(t-\xi)^\beta]_{\beta,[\xi,\xi+r]}$.
The limit exists by monotonicity. The seminorm inequality shows
$|\Phi_b(L)-\Phi_b(K)|\le|L-K|$ and
$\Phi_b(L)\ge|L|-[b]_\beta$. Therefore $\Phi_b$ is continuous
and coercive, so its minimum is attained.

For the lower bound choose $f_-$ supported immediately to the left
of $\xi$ with
$f_-(s)=(\xi-s)^\beta$ on $[\xi-\eta,\xi]$ and $f_-(s)=0$
on $[\xi,1]$, using a Lipschitz cutoff on
$[\xi-2\eta,\xi-\eta]$.
Then $f_-\in AC\cap C^\beta\subset A_\beta$.
For $t\in[\xi,\xi+r]$ and all sufficiently small $r$,
\[
(H_bf_-)(t)=b(t),
\]
whereas for any
$S=M_v+\sum_j M_{p_j}C_{\phi_j}\in\mathscr R_{\beta,n}$,
\[
(Sf_-)(t)=p_1(t)(t-\xi)^\beta.
\]
Indeed, $f_-(t)=0$, $\phi_2(t)=t$, the other branch images lie
outside the support of $f_-$, and
$\phi_1(t)=2\xi-t$. It follows that
\[
[b-p_1(t)(t-\xi)^\beta]_{\beta,[\xi,\xi+r]}
\le\|f_-\|_\beta\,\|H_b-S\|.
\]
Since $p_1\in A_\beta\subset C^\beta$, putting $L=p_1(\xi)$
we have
\[
[(p_1-p_1(\xi))(t-\xi)^\beta]_{\beta,[\xi,\xi+r]}
\le 2r^\beta[p_1]_\beta\longrightarrow0.
\]
Taking $r\downarrow0$ and then the infimum over $S$ yields
$\Omega^+_{\beta,\xi}(b)\le\|f_-\|_\beta
\operatorname{dist}(H_b,\mathscr R_{\beta,n})$.

For the upper bound choose $L$ attaining the minimum. Select
$\chi_0\in C^\infty([0,1])$ supported near $\xi$, equal to one on
$[\xi,\xi+\eta]$, and supported on the left only in $I_1$.
For $t\ge\xi$, set
$d(t)=b(t)-L\chi_0(t)h_\xi(t)$; then $d(\xi)=0$.
Let $\chi_r$ be a Lipschitz cutoff which equals one on
$[\xi,\xi+r]$ and zero for $t\ge\xi+2r$. Define a coefficient
$q_r\in A_\beta$ by
\[
q_r(t)=\begin{cases}
L\chi_0(t),&t\le\xi,\\[2pt]
L\chi_0(t)+(1-\chi_r(t))\dfrac{d(t)}{h_\xi(t)},&t>\xi.
\end{cases}
\]
For fixed $r>0$, the quotient is taken away from $\xi$ and
extended by zero near $\xi$, so $q_r\in A_\beta$. Consequently
$S_r=M_{-q_r}+M_{q_r}C_{\phi_1}\in\mathscr R_{\beta,n}$.
On $I_1$ the operator $C_{\phi_1}-I$ vanishes, and on the other
branches the coefficient $q_r$ is zero. Therefore
\[
H_b-S_r=H_{\widetilde d_r},\qquad
\widetilde d_r(t)=\begin{cases}
\chi_r(t)d(t),&t>\xi,\\0,&t\le\xi.
\end{cases}
\]
The standard H\"older cutoff estimate, using $d(\xi)=0$, gives
for small $r$
\[
\|\widetilde d_r\|_\beta
\le C_\beta[d]_{\beta,[\xi,\xi+2r]}
=C_\beta[b-L(t-\xi)^\beta]_{\beta,[\xi,\xi+2r]},
\]
because $\chi_0=1$ near $\xi$. The preceding multiplier estimate
then yields
$\|H_b-S_r\|\le C'_{\beta,n}
[b-L(t-\xi)^\beta]_{\beta,[\xi,\xi+2r]}$.
Letting $r\downarrow0$ proves the upper bound.

To see that the inclusions are strict, fix $0<\alpha<\beta$ and
choose a smooth cutoff $\rho$, equal to one near $\xi$ and supported
in $[\xi-\eta,\xi+\eta]$. Put
\[
b_1(t)=\rho(t)(t-\xi)_+^{\beta+\alpha}.
\]
Then $b_1\in AC\cap C^\beta$, the one-sided defect vanishes, and
$H_{b_1}\in\overline{\mathscr R_{\beta,n}}^{\|\cdot\|}$.
But $H_{b_1}\notin\mathscr R_{\beta,n}$: testing on the distinct
fiber points on $I_2$ shows that a regular representation would
force $p_1(t)=\rho(t)(t-\xi)^\alpha$ for $t>\xi$ close to $\xi$,
contrary to $p_1\in C^\beta$.

Finally, let
\[
b_2(t)=\begin{cases}
\rho(t)(t-\xi)^\beta\sin(\log(1/(t-\xi))),&t>\xi,\\
0,&t\le\xi.
\end{cases}
\]
Its derivative on the right is $O((t-\xi)^{\beta-1})$, so
$b_2\in AC\cap C^\beta\subset A_\beta$. Since its oscillatory
factor approaches both $1$ and $-1$ arbitrarily near $\xi$,
$\Phi_{b_2}(L)\ge\max\{|1-L|,|-1-L|\}\ge1$ for every $L$.
The lower distance bound shows
$H_{b_2}\notin\overline{\mathscr R_{\beta,n}}^{\|\cdot\|}$.
\end{proof}

\begin{example}\label{ex:three-branches}
For $n=3$,
\[
u_3(t)=\begin{cases}3t,&0\le t\le1/3,\\
2-3t,&1/3\le t\le2/3,\\
3t-2,&2/3\le t\le1,
\end{cases}
\]
and the inverse branches are
$\tau_1(y)=y/3$, $\tau_2(y)=(2-y)/3$, and
$\tau_3(y)=(2+y)/3$. For every
$T=M_v+\sum_{j=1}^3M_{p_j}C_{\phi_j}$,
Theorem~\ref{thm:finite-spectrum} gives the spectrum of $T$ as the
union of the spectra of the explicit $3\times3$ matrices
\[
\left(v(\tau_i(y))\delta_{ij}+p_j(\tau_i(y))\right)_{i,j=1}^3.
\]
At $\xi=1/3$, Theorem~\ref{thm:fold-distance} produces a
norm-limit of regular commuting operators with a genuinely
non-$\beta$-H\"older coefficient, and also a bounded commuting
operator that is not such a norm-limit.

For an explicit spectral calculation, put
$W=\sum_{j=1}^3 C_{\phi_j}$. The branch identities give $W^2=3W$,
and hence, for $\lambda\ne0$,
$\sigma(I+\lambda W)=\{1,1+3\lambda\}$.
The operator is invertible precisely when $\lambda\ne-1/3$,
and in that case
$(I+\lambda W)^{-1}=I-\lambda W/(1+3\lambda)$.
(The same formula holds for $\lambda=0$, when the spectrum is
$\{1\}$.)
\end{example}

\section{Complete finite-branch norm closure}
\label{sec:complete-finite}
We now determine the norm closure of the regular finite-branch algebra. The class $\mathfrak M_{\beta,n}$ of symbols introduced below turns out to be the full commutant of $M_{u_n}$ (Corollary~\ref{cor:class-full}). Its regular norm closure is characterized by two endpoint defects.
We use the local weighted multiplier estimate from
Lemma~\ref{lem:multiplier-bound}, also valid on one-sided intervals:
if $b,g\in A_\beta$ vanish at $e\in\{0,1\}$, then
\begin{equation}\label{weighted}
 \left\|\frac{bg}{h_e}\right\|_{A_\beta}
 \le C_\beta\|b\|_{A_\beta}\|g\|_{A_\beta},
 \qquad h_e(y)=|y-e|^\beta,
\end{equation}
with the quotient set equal to zero at $e$.
For matrix-valued functions this estimate is applied componentwise.
The same notation $A_\beta(J)$ is used for the restricted-variation
algebra on a subinterval $J$, with its natural norm.

For clarity, the one-sided version of \eqref{weighted} follows from the
same increment estimate as Lemma~\ref{lem:multiplier-bound}.
If $x,y$ lie on the same side of $e$, ordered so that
$0<h_e(x)\le h_e(y)$, and $b(e)=g(e)=0$, then
\[
\left|\frac{b(y)g(y)}{h_e(y)}-\frac{b(x)g(x)}{h_e(x)}\right|
\le [g]_\beta|b(y)-b(x)|+[b]_\beta|g(y)-g(x)|
 +[b]_\beta[g]_\beta|h_e(y)-h_e(x)|.
\]
The estimate also holds with one point equal to $e$. Since
$h_e\in AC\cap C^\beta\subset A_\beta$, summing over the
interval families defining restricted variation gives \eqref{weighted},
including its $A_\beta$ conclusion and norm estimate.
\newcommand{\Rn}{\mathscr R_{\beta,n}}
\newcommand{\Xn}{X_{\beta,n}}
\newcommand{\dist}{\operatorname{dist}}
For $e\in\{0,1\}$, let $W_e\subset\mathbb C^n$ be the space of endpoint values compatible with the folds:
\[
 W_e=\{z\in\mathbb C^n:\ z_j=z_{j+1}\ \text{if }u_n(j/n)=e,\ 1\le j<n\}.
\]
Here $\mathbf e_1,\ldots,\mathbf e_n$ are the standard coordinate vectors of
$\mathbb C^n$, and $M_n(\mathbb C)$ denotes the space of complex
$n\times n$ matrices. Let $P_e$ be the \emph{orthogonal} projection
onto $W_e$ in the usual Euclidean inner product, and put $Q_e=I-P_e$.
The projection $Q_e$ selects the endpoint directions which vanish for
vectors in $W_e$. Put
\[
 \mathcal L_e=\bigl\{L\in M_n(\mathbb C):\ L=LQ_e,\quad
 (\mathbf e_j^{\mathsf T}-\mathbf e_{j+1}^{\mathsf T})L
 \in\mathbb C(\mathbf e_j^{\mathsf T}-\mathbf e_{j+1}^{\mathsf T})
 \ \text{whenever }u_n(j/n)=e\bigr\}.
\]
Thus $\mathcal L_e$ contains the admissible endpoint matrices acting
on the $Q_e$ directions. To make the simultaneous-fold condition
explicit, write $r_j=\mathbf e_j^{\mathsf T}-\mathbf e_{j+1}^{\mathsf T}$
and $\mathcal J_e=\{j:u_n(j/n)=e\}$. Endpoint preservation alone gives
$r_jA(e)P_e=0$, and therefore only
\[
 r_jA(e)\in W_e^\perp
       =\operatorname{span}\{r_k:k\in\mathcal J_e\}
       \qquad(j\in\mathcal J_e).
\]
It allows a row difference at one fold to involve the difference
vectors of other folds at the same level. By contrast, the regular
symbol condition requires $r_jA(e)\in\mathbb C r_j$ for
\emph{each} $j\in\mathcal J_e$ separately. Accordingly,
$L\in\mathcal L_e$ cannot mix these fold-difference directions.
This distinction is essential in the upper and lower distance estimates,
not merely in the choice of notation. The matrix norms and
matrix-valued H\"older seminorms below
are taken entrywise, using the maximum of the scalar norms; any
fixed equivalent finite-dimensional matrix norm gives the same closure.

\begin{lemma}\label{branchmodel}
The map $Uf(y)=(f(\tau_1(y)),\ldots,f(\tau_n(y)))^{\mathsf T}$ is a bounded linear isomorphism of $A_\beta([0,1])$ onto
\[
 \Xn=\{F\in A_\beta([0,1];\mathbb C^n):\ F(0)\in W_0,\ F(1)\in W_1\},
\]
where the vector-valued norm is the maximum of the component $A_\beta$ norms. Under this isomorphism, $\Rn$ consists \emph{exactly} of multiplication operators $T_A F(y)=A(y)F(y)$ with
\begin{equation}\label{regularsymbol}
 A\in M_n(A_\beta([0,1])),\qquad \text{For every fold }j,\; A_{j\bullet}(e)-A_{j+1,\bullet}(e)\in\mathbb C(\mathbf e_j^{\mathsf T}-\mathbf e_{j+1}^{\mathsf T}),\quad e=u_n(j/n).
\end{equation}
\end{lemma}
\begin{proof}
The boundedness of $U$ follows from affine composition. Conversely, components $F_i\in A_\beta$ with the stated matching values can be glued on the $n$ physical branches. The H\"older estimate follows by splitting a pair of points at the finitely many joining points. For the restricted variation condition, split every admissible interval at the same joining points. Total length does not increase, the $\beta$-power budget increases by at most $n^{1-\beta}$, and the resulting intervals on each branch have disjoint interiors. Applying the defining variation condition on each of the finitely many branches proves that the glued function belongs to $A_\beta$. This also gives an inverse norm bound.

A regular operator has on the fibers the matrix
\[
 A_{ij}(y)=v(\tau_i(y))\delta_{ij}+p_j(\tau_i(y));
\]
its entries belong to $A_\beta$. At the physical fold $x_j=j/n$
with level $e=u_n(x_j)$ we have
$\tau_j(e)=\tau_{j+1}(e)=x_j$. Hence the terms
$p_k(x_j)$ cancel in the difference of rows $j$ and $j+1$, giving
\[
 A_{j\bullet}(e)-A_{j+1,\bullet}(e)
   =v(x_j)(\mathbf e_j^{\mathsf T}-\mathbf e_{j+1}^{\mathsf T}).
\]
This is the \emph{individual} row-difference condition at each fold,
including when several folds have the same endpoint level. It implies
$A(e)W_e\subset W_e$, but the latter inclusion alone does not recover
these individual identities.

Conversely, take a matrix satisfying \eqref{regularsymbol}. At the joining point $j/n$ let $e=u_n(j/n)$. By the assumed local row-difference condition,
\[
 A_{j\bullet}(e)-A_{j+1,\bullet}(e)
       =\lambda_j(\mathbf e_j^{\mathsf T}-\mathbf e_{j+1}^{\mathsf T})
\]
for some scalar $\lambda_j$. Choose a smooth scalar function $v$ on $[0,1]$ with $v(j/n)=\lambda_j$ for every joining point. On branch $I_i$ define $p_k$ by
\[
 p_k(\tau_i(y))=A_{ik}(y)-v(\tau_i(y))\delta_{ik}.
\]
The row-difference identity makes all $p_k$ continuous across each joining point. Each branch restriction belongs to $A_\beta$ by affine composition, and the preceding finite-gluing argument gives $p_k\in A_\beta([0,1])$. Hence $A$ represents $M_v+\sum_kM_{p_k}C_{\phi_k}$.
\end{proof}

Write $J_0=[0,\eta]$, $J_1=[1-\eta,1]$ for a fixed $0<\eta<1/4$, and $J_e(r)=\{y\in[0,1]:|y-e|\le r\}$. Let $A:(0,1)\to M_n(\mathbb C)$ be a matrix function, and assume:
\begin{enumerate}
\item[(S1)] Every entry of $A$ belongs to $A_\beta(K)$ for each closed interval $K\subset(0,1)$.
\item[(S2)] On $J_e\setminus\{e\}$ the matrices
\[
 V_e(y)=A(y)P_e,\qquad B_e(y)=h_e(y)A(y)Q_e,
 \qquad h_e(y)=|y-e|^\beta,
\]
extend to $M_n(A_\beta(J_e))$, with $B_e(e)=0$ and $V_e(e)W_e\subset W_e$.
\end{enumerate}
Here $V_e$ is the part acting on compatible endpoint values, while $B_e$ is the weighted part acting on complementary directions. The notation $J_e(r)$ always refers to an interval in the level variable $y$.
For such a symbol define the \emph{endpoint defects}
\begin{equation}\label{defect}
 \Omega_e(B_e)=\min_{L\in\mathcal L_e}\ \lim_{r\downarrow0}
      [B_e-h_eL]_{\beta,J_e(r)},\quad e=0,1.
\end{equation}
The minimum exists: the limiting seminorm is Lipschitz in $L$ and is bounded below by $\|L\|- [B_e]_{\beta,J_e}$, since $[h_e]_{\beta,J_e(r)}=1$. Moreover $\Omega_e(B_e)$ does not depend on the particular choice of $\eta$.

The defect has a useful precise interpretation as a distance in a
space of local H\"older germs. This also separates the general distance
method from the operator-specific estimates used below.

\begin{lemma}\label{lem:germ-reduction}
Fix $e\in\{0,1\}$, and let
\[
 \mathcal G_e=\{B\in A_\beta(J_e;M_n(\mathbb C)):\ B(e)=0\}.
\]
For $B\in\mathcal G_e$, set
\[
 \rho_e(B)=\lim_{r\downarrow0}[B]_{\beta,J_e(r)},\qquad
 \mathcal N_e=\{B\in\mathcal G_e:\rho_e(B)=0\}.
\]
Then $\rho_e$ is a seminorm, $\mathcal N_e$ is a linear subspace, and
$\rho_e$ induces a norm on $\mathcal G_e/\mathcal N_e$.
If $[B]$ denotes its coset and
$\mathcal H_e=\{[h_eL]:L\in\mathcal L_e\}$, then
$\mathcal H_e$ is finite dimensional and closed, and
\[
 \Omega_e(B)=\operatorname{dist}_{\mathcal G_e/\mathcal N_e}
             ([B],\mathcal H_e).
\]
This identity is a reduction of the \emph{defect} to a quotient-space
distance; it does not by itself imply the operator-norm bounds in
Theorem~\ref{fullclosure}.
\end{lemma}
\begin{proof}
The local seminorms decrease as $r\downarrow0$, so their limit exists.
Their triangle inequality and homogeneity pass to the limit. Hence
$\rho_e$ is a seminorm and the quotient carries the asserted norm.
We also verify completeness in this norm. Let
\[
 d_e(B)=\inf_{N\in\mathcal N_e}\|B-N\|_{A_\beta(J_e)}.
\]
Since $\rho_e(N)=0$ for $N\in\mathcal N_e$ and
$\rho_e(H)\le\|H\|_{A_\beta(J_e)}$, we have
$\rho_e(B)\le d_e(B)$. To prove the converse, fix
$0<r<\eta/2$ and choose a smooth cutoff $\chi_r$ that equals
one on $J_e(r)$, vanishes outside $J_e(2r)$, and satisfies
$\operatorname{Lip}(\chi_r)\le C/r$. Then
$(1-\chi_r)B\in\mathcal N_e$. Since $B(e)=0$,
\[
 \|B\|_{\infty,J_e(2r)}
 \le (2r)^\beta[B]_{\beta,J_e(2r)},
 \qquad [\chi_r]_{\beta,J_e}\le C_\beta r^{-\beta}.
\]
Applying the H\"older product estimate on the support of
$\chi_r$, and the cutoff estimate across its boundary, gives
\[
 d_e(B)\le\|\chi_r B\|_{A_\beta(J_e)}
 \le C_\beta[B]_{\beta,J_e(2r)}.
\]
Here $\chi_r B\in A_\beta(J_e)$ because $A_\beta$ is an
algebra containing smooth functions. Letting $r\downarrow0$ yields
\[
 \rho_e(B)\le d_e(B)\le C_\beta\rho_e(B).
\]
Now $\mathcal G_e$ and $\mathcal N_e$ are closed in
$A_\beta(J_e;M_n)$; hence their usual Banach quotient is
complete. The displayed norm equivalence proves that
$\mathcal G_e/\mathcal N_e$ is complete also in the norm
induced by $\rho_e$.
For a constant matrix $L$, the identity
$[h_eL]_{\beta,J_e(r)}=\|L\|_{\mathrm{entry}}$ holds for all sufficiently
small $r$ (with the maximum-entry matrix norm). Consequently
$L\mapsto[h_eL]$ embeds $\mathcal L_e$ isometrically into the quotient,
so its image is finite dimensional and closed. The distance identity
is exactly the definition of $\Omega_e$.
\end{proof}

Lemma~\ref{lem:germ-reduction} expresses the endpoint defect as a quotient-space distance, as in general distance-formula methods including Perfekt~\cite{Perfekt2013} and Angrisani et al.~\cite{Angrisani2020}. The operator-norm bounds in Theorem~\ref{fullclosure} additionally use endpoint-compatible regular symbols and localized test vectors.

\begin{lemma}\label{bounded}
Every symbol satisfying \textup{(S1)--(S2)} defines a bounded operator $T_A$ on $\Xn$. Its endpoint action is given by
\[
 (T_AF)(e)=V_e(e)F(e),\qquad e=0,1.
\]
\end{lemma}
\begin{proof}
Near $e$ write $F=P_eF+Q_eF$. Since $F(e)\in W_e$, the vector $Q_eF$ vanishes at $e$. For $y\ne e$,
\[
 A(y)F(y)=V_e(y)F(y)+\frac{B_e(y)}{h_e(y)}Q_eF(y).
\]
The first term belongs to $A_\beta(J_e;\mathbb C^n)$ by the Banach algebra product inequality. Each component of the second belongs to $A_\beta(J_e)$ by \eqref{weighted} and vanishes at $e$. The endpoint value therefore lies in $W_e$ by \textup{(S2)}. On compact intervals away from $0,1$, matrix multiplication is bounded by \textup{(S1)}. Finite gluing over the three subintervals proves $T_AF\in\Xn$ and yields a uniform operator norm bound.
\end{proof}

\begin{lemma}\label{lem:symbol-recovery}
For $T_A$ as in Lemma~\ref{bounded}, fix a closed interval
$K\subset(0,1)$. There is a constant $C_{\beta,n,\eta,K}$ such that
\[
 \|A\|_{A_\beta(K)}+
 \sum_{e=0,1}\bigl(\|V_e\|_{A_\beta(J_e)}+
                          \|B_e\|_{A_\beta(J_e)}\bigr)
 \le C_{\beta,n,\eta,K}\,\|T_A\|_{\mathcal B(X_{\beta,n})}.
\]
The norms of matrix-valued functions are the maximum of the
entrywise $A_\beta$ norms.
\end{lemma}
\begin{proof}
Choose smooth functions $\chi_e$ equal to one on $J_e$ and zero
near the opposite endpoint. For $w\in W_e$, the test function
$F_w=\chi_e w$ belongs to $X_{\beta,n}$, and
$(T_AF_w)|_{J_e}=V_e w$. Choose a fixed basis of $W_e$.
The $A_\beta$ norms of the resulting output functions control
all columns of $V_e=A P_e$, up to a constant depending only on $n$
and the fixed cutoffs.
Likewise, for a fixed basis $q$ of $\operatorname{ran}Q_e$,
$F_q=\chi_e h_e q$ lies in $X_{\beta,n}$ and
$(T_AF_q)|_{J_e}=B_e q$. These outputs recover $B_e=B_e Q_e$
with the same type of bound.
Finally, select a smooth cutoff $\theta_K$ supported in $(0,1)$
and equal to one on $K$. For the standard basis vectors
$\mathbf e_j$, the functions $\theta_K\mathbf e_j$ belong to
$X_{\beta,n}$, and
$(T_A(\theta_K\mathbf e_j))|_K=A\mathbf e_j|_K$.
Boundedness of restriction to a subinterval and finiteness of the
selected bases give the stated estimate.
\end{proof}

\begin{lemma}\label{lem:closure-cutoff}
Let $e\in\{0,1\}$, $0<r<\eta/4$, and let $\chi_{e,r}$
be a Lipschitz cutoff which equals one on $J_e(r)$, vanishes outside
$J_e(2r)$, and satisfies $\operatorname{Lip}(\chi_{e,r})\le C/r$.
If $d\in A_\beta(J_e;M_n(\mathbb C))$ and $d(e)=0$, then
\begin{equation}\label{eq:uniform-cutoff}
 \|\chi_{e,r}d\|_{A_\beta(J_e)}
 \le C_{\beta,n}[d]_{\beta,J_e(2r)}.
\end{equation}
If $F\in X_{\beta,n}$, then, with $Q_eF(e)=0$,
\begin{equation}\label{eq:uniform-endpoint-operator}
 \left\|\frac{\chi_{e,r}d}{h_e}Q_eF\right\|_{A_\beta(J_e)}
 \le C_{\beta,n}[d]_{\beta,J_e(2r)}\|F\|_{X_{\beta,n}}.
\end{equation}
The two expressions are extended by zero where necessary; the
constants are independent of $r,d$, and $F$.
\end{lemma}
\begin{proof}
Put $D_r=[d]_{\beta,J_e(2r)}$. Since $d(e)=0$,
$\sup_{J_e(2r)}\|d(y)\|\le(2r)^\beta D_r$.
For $x,y\in J_e$, the increment of $\chi_{e,r}d$ is bounded by
$D_r|x-y|^\beta+C r^{\beta-1}D_r|x-y|$ when
$|x-y|\le r$, using the cutoff Lipschitz bound and the
vanishing outside $J_e(2r)$. This is at most
$C_\beta D_r|x-y|^\beta$. For $|x-y|>r$ use the
supremum bound $C_\beta r^\beta D_r$ at both points.
This proves \eqref{eq:uniform-cutoff}, including its supremum term.
Membership in $A_\beta$ follows because $d\in A_\beta$ and the
cutoff is Lipschitz. Every component of $Q_eF$ belongs to
$A_\beta(J_e)$ and vanishes at $e$. Applying the one-sided
weighted multiplier estimate \eqref{weighted} componentwise to
$\chi_{e,r}d$ and $Q_eF$ gives
\eqref{eq:uniform-endpoint-operator}. Because the cutoff vanishes
before the other end of $J_e$, extension by zero has a fixed norm
bound independent of $r$. This also follows directly by splitting
crossing intervals at that end; their total length does not
increase and their $\beta$-power budget increases by at most
$2^{1-\beta}$.
\end{proof}

\begin{theorem}\label{fullclosure}
Let $T_A$ be a bounded fiber-matrix operator on $\Xn$ with a symbol satisfying \textup{(S1)--(S2)}. There are constants $0<c_{\beta,n,\eta}\le C_{\beta,n,\eta}<\infty$, independent of $A$, such that
\begin{equation}\label{mainbound}
 c_{\beta,n,\eta}\max\{\Omega_0(B_0),\Omega_1(B_1)\}
 \le \dist\bigl(T_A,\Rn\bigr)
 \le C_{\beta,n,\eta}\max\{\Omega_0(B_0),\Omega_1(B_1)\}.
\end{equation}
Here operator distance is taken in $\mathcal B(\Xn)$, equivalently in $\mathcal B(A_\beta)$ under $U$. In particular,
\begin{equation}\label{complete}
\boxed{\quad\overline{\Rn}^{\,\|\cdot\|}
=\{T_A:\ A\text{ satisfies \textup{(S1)--(S2)} and }
\Omega_0(B_0)=\Omega_1(B_1)=0\}.\quad}
\end{equation}

\end{theorem}
\begin{proof}
We first prove the upper estimate. Choose matrices $L_e\in\mathcal L_e$ attaining the minima in \eqref{defect}, and smooth cutoffs $\chi_{e,r}$, supported in $J_e(2r)$ and equal to one in $J_e(r)$, with Lipschitz constants at most $C/r$. Take $r<\eta/4$ so their supports are disjoint. Set, on $J_e\setminus\{e\}$,
\[
 A_r(y)=A(y)-\frac{\chi_{e,r}(y)}{h_e(y)}
                 (B_e(y)-h_e(y)L_e),
\]
and let $A_r=A$ outside the two cutoff supports. Since $B_e(y)=h_e(y)A(y)Q_e$, this becomes near $e$
\[
 A_r(y)=V_e(y)+L_e+(1-\chi_{e,r}(y))
                        \bigl(A(y)Q_e-L_e\bigr).
\]
The last term vanishes near $e$ and is $A_\beta$ away from $e$. Thus $A_r$ extends to a matrix in $M_n(A_\beta([0,1]))$. Also $A_r(e)=V_e(e)+L_e$. Fix any fold
$j\in\mathcal J_e$, with $r_j$ as defined above. Since
$V_e(e)=V_e(e)P_e$ and $V_e(e)W_e\subset W_e$, its range is
contained in $W_e$; thus $r_jV_e(e)=0$. Independently for this
particular $j$, the definition $L_e\in\mathcal L_e$ gives
$r_jL_e=\lambda_{e,j}r_j$ for some scalar $\lambda_{e,j}$. Therefore
\[
 r_jA_r(e)=r_jV_e(e)+r_jL_e=\lambda_{e,j}r_j
       \qquad\text{for every }j\in\mathcal J_e.
\]
This verifies \eqref{regularsymbol} fold by fold, even when several
folds occur at the same level. Merely verifying
$A_r(e)W_e\subset W_e$ would not be enough.
By Lemma~\ref{branchmodel}, $T_{A_r}\in\Rn$.

Write $d_e=B_e-h_eL_e$, which vanishes at $e$. Its cutoff satisfies
\[
 \|\chi_{e,r}d_e\|_{A_\beta(J_e)}
 \le C_\beta[d_e]_{\beta,J_e(2r)};
\]
this follows from the ordinary H\"older cutoff estimate and $|d_e(y)|\le [d_e]_{\beta,J_e(2r)}|y-e|^\beta$. Lemma~\ref{lem:closure-cutoff} applies to the matrix $d_e$ and to $Q_eF$; in particular the bound is uniform as $r\downarrow0$. Therefore
\[
 \|T_A-T_{A_r}\|\le C_{\beta,n}
       \max_{e=0,1}[d_e]_{\beta,J_e(2r)}.
\]
Letting $r\downarrow0$ proves the right-hand inequality in \eqref{mainbound}.

For the lower bound fix an endpoint $e$ and a cutoff $\chi_e$ which equals one on $J_e$ and vanishes near the other endpoint. For each basis vector $q$ of $\operatorname{ran}Q_e$, the vector function
\[
 F_q(y)=\chi_e(y)h_e(y)q
\]
belongs to $\Xn$, and its norm is bounded by a constant depending only on the fixed cutoff and $n,\beta$. On $J_e$, $T_AF_q=B_e q$. For a regular symbol $D$ satisfying \eqref{regularsymbol}, the corresponding $B_e^D=h_eDQ_e$ has
\[
 \Omega_e(B_e^D)=0,
\]
for the following two separate reasons. First, at every
$j\in\mathcal J_e$ the regular-symbol identity gives
$r_jD(e)=\lambda_jr_j$. As $r_jQ_e=r_j$, the matrix
$L=D(e)Q_e$ satisfies $L=LQ_e$ and
$r_jL=\lambda_jr_j$; hence $L\in\mathcal L_e$.
Second, since $D$ is $\beta$-H\"older, the ordinary product
seminorm estimate yields
\[
 [h_e(D-D(e))Q_e]_{\beta,J_e(r)}
 \le C_n r^\beta[D]_{\beta,J_e(r)}
 \longrightarrow0
\qquad(r\downarrow0).
\]
Here $C_n$ only accounts for the fixed matrix dimensions and the
entrywise norm. Thus the admissible
matrix $D(e)Q_e$, not a general endpoint-preserving matrix, removes
the local germ of $B_e^D$. Moreover, for each fixed basis vector $q$
of $\operatorname{ran}Q_e$, the test function $F_q=\chi_eh_eq$
gives
$\|(B_e-B_e^D)q\|_{A_\beta(J_e)}
\le C_{\beta,n}\|T_A-T_D\|$. Summing over the finitely many
basis vectors controls the matrix-valued germ seminorm. Using the
triangle inequality for the defect gives
\[
 \Omega_e(B_e)\le C_{\beta,n}\|T_A-T_D\|.
\]
Take the infimum over $T_D\in\Rn$ and then the maximum over $e$
to obtain the left-hand inequality. The restriction
$D(e)Q_e\in\mathcal L_e$ is essential: replacing
$\mathcal L_e$ by all matrices satisfying only endpoint preservation
would erase the cross-fold obstruction and invalidate this lower-bound
argument for the distance to the \emph{regular} algebra.

It remains to prove that the displayed class is \emph{all} of the closure. Let $T$ be a norm limit of regular operators $T_{A_m}$. At every $y\in(0,1)$ the evaluation $E_y:\Xn\to\mathbb C^n$, $E_yF=F(y)$, is onto: use smooth test functions supported in $(0,1)$ that interpolate a prescribed fiber vector. Since $E_yT_{A_m}=A_m(y)E_y$, the matrices $A_m(y)$ converge to a unique $A(y)$ with $E_yT=A(y)E_y$. To see \textup{(S1)}, use cutoffs supported in $(0,1)$ times the constant coordinate vectors; their $T$-images agree locally with the columns of $A(y)$ and belong to $A_\beta$.

For \textup{(S2)}, let $\chi_e$ be a cutoff equal to one near $e$ and zero near the opposite endpoint. For each $w\in W_e$, the vector function $F_w(y)=\chi_e(y)w$ belongs to $\Xn$. Its $T$-image agrees with $A(y)w$ near $e$, so $A(y)P_e$ extends into $A_\beta$ and has an endpoint value mapping $W_e$ into $W_e$. For each $q\in\operatorname{ran}Q_e$, the test function $F_q=\chi_eh_eq$ used above shows that $h_eA(y)q$ extends into $A_\beta$. Furthermore $(T_{A_m}F_q)(e)=0$ for every $m$, so by norm convergence $(TF_q)(e)=0$. Hence $B_e(e)=0$. The lower-bound argument, now applied to $T$ and $T_{A_m}$, gives $\Omega_e(B_e)=0$ because $\|T-T_{A_m}\|\to0$. Thus every member of the closure has a symbol in the right side of \eqref{complete}. The upper-bound argument proves the converse.
\end{proof}

\begin{corollary}\label{cor:operator-germ}
Let $\mathfrak M_{\beta,n}$ be the space of all bounded fiber-matrix
operators $T_A$ on $X_{\beta,n}$ whose symbols satisfy
\textup{(S1)--(S2)}, with the operator norm. With the notation of
Lemma~\ref{lem:germ-reduction}, put
\[
 \mathcal Y_e=(\mathcal G_e/\mathcal N_e)/\mathcal H_e,
 \qquad e=0,1.
\]
Then $\mathfrak M_{\beta,n}$ is a Banach space and
\[
 \Gamma:\mathfrak M_{\beta,n}\longrightarrow
       \mathcal Y_0\oplus_\infty\mathcal Y_1,
 \qquad
 \Gamma(T_A)=\bigl([B_0]+\mathcal H_0,
                         [B_1]+\mathcal H_1\bigr)
\]
is a bounded linear map satisfying
\[
 \ker\Gamma=\overline{\mathscr R_{\beta,n}}^{\,\|\cdot\|}.
\]
Its induced map
\[
 \widetilde\Gamma:
 \mathfrak M_{\beta,n}/\overline{\mathscr R_{\beta,n}}^{\,\|\cdot\|}
 \longrightarrow\Gamma(\mathfrak M_{\beta,n})
\]
is a Banach-space isomorphism onto a closed subspace of
$\mathcal Y_0\oplus_\infty\mathcal Y_1$. More precisely,
\[
 c_{\beta,n,\eta}\|\widetilde\Gamma([T_A])\|
 \le \|[T_A]\|
 \le C_{\beta,n,\eta}\|\widetilde\Gamma([T_A])\|.
\]

\end{corollary}
\begin{proof}
The local seminorm $\rho_e$ is bounded by the $A_\beta$ norm;
therefore $\mathcal N_e$ is closed in $\mathcal G_e$.
The subspace $\mathcal H_e$ is finite dimensional, hence closed,
so each $\mathcal Y_e$ is a Banach space.

To see completeness of $\mathfrak M_{\beta,n}$, consider an
operator-norm convergent sequence of its members. Interior fiber
evaluations and cutoff test functions recover the limiting symbol
and condition \textup{(S1)}. For $w\in W_e$, the test functions
$\chi_e w$ recover $A P_e$ and its compatible endpoint value.
For $q\in\operatorname{ran}Q_e$, the functions $\chi_e h_e q$
recover $h_e A Q_e$ with endpoint value zero.
These are exactly the tests used in the last paragraph of the
proof of Theorem~\ref{fullclosure}, and establish \textup{(S2)}
without invoking the vanishing-defect condition.
Thus $\mathfrak M_{\beta,n}$ is closed in $\mathcal B(X_{\beta,n})$.

The two components of $\Gamma$ are linear and, by
Lemma~\ref{lem:germ-reduction},
\[
 \|\Gamma(T_A)\|=\max_{e=0,1}\Omega_e(B_e).
\]
Theorem~\ref{fullclosure} now gives
\[
 c_{\beta,n,\eta}\|\Gamma(T_A)\|
 \le\operatorname{dist}(T_A,\mathscr R_{\beta,n})
 \le C_{\beta,n,\eta}\|\Gamma(T_A)\|.
\]
Since the regular algebra contains zero, the first inequality
also proves continuity of $\Gamma$. Its kernel is the operator-norm
closure by \eqref{complete}. The displayed inequalities descend
to the quotient and show that $\widetilde\Gamma$ has a bounded
inverse on its range. The domain is complete, so the range is closed.
\end{proof}

\noindent\textit{Remark.} Lemma~\ref{lem:germ-reduction} expresses the endpoint defect as a quotient-space distance of the type studied in general distance-formula methods, including Perfekt~\cite[Theorem~1.3]{Perfekt2013} and Angrisani et al.~\cite{Angrisani2020}. It does not by itself imply the operator-norm bounds in Theorem~\ref{fullclosure}. The additional input there is the comparison with operator-norm distance: the upper bound uses endpoint-compatible regular symbols, and the lower bound uses test vectors $\chi_e h_e q$.

\begin{proposition}\label{crossfold}
For $n=4$ there is a bounded fiber-matrix operator commuting with $M_{u_4}$ whose symbol is a constant matrix, but which does not belong to $\overline{\mathscr R_{\beta,4}}$. In particular, endpoint preservation $A(e)W_e\subset W_e$ alone is not the correct regular-symbol condition when several folds share a level.
\end{proposition}
\begin{proof}
Here $W_1=\{z:z_1=z_2,\ z_3=z_4\}$, while $W_0=\{z:z_2=z_3\}$. Let $A$ be the constant $4\times4$ matrix whose first row is $(0,0,1,-1)$ and whose other rows are zero. Then $A(W_1)=\{0\}$ and $A(W_0)\subset W_0$, so $F\mapsto AF$ is bounded on $X_{\beta,4}$. But the first and second rows of $A$ differ by $\mathbf e_3^{\mathsf T}-\mathbf e_4^{\mathsf T}$, which is not a multiple of $\mathbf e_1^{\mathsf T}-\mathbf e_2^{\mathsf T}$. For the endpoint $e=1$, $AP_1=0$ and $B_1(y)=h_1(y)AQ_1=h_1(y)A$. Since $[h_1]_{\beta,J_1(r)}=1$,
\[
 \Omega_1(B_1)=\min_{L\in\mathcal L_1}\|A-L\|=\frac12.
\]
Indeed, the first two rows of every $L\in\mathcal L_1$ have
identical third entries. Since $A_{13}-A_{23}=1$, the
maximum-entry norm of $A-L$ is at least $1/2$. Equality is
attained when the first two rows of $L$ both equal
$(0,0,1/2,-1/2)$ and the remaining rows vanish; this matrix
belongs to $\mathcal L_1$.
We can estimate the operator-norm distance directly. Put
\[
 q_1=\frac{\mathbf e_1-\mathbf e_2}{\sqrt2},\qquad
 q_3=\frac{\mathbf e_3-\mathbf e_4}{\sqrt2},\qquad
 F(y)=y(1-y)^\beta q_3.
\]
Then $F\in X_{\beta,4}$ vanishes at both endpoints and
$\|F\|_{X_{\beta,4}}\le3/\sqrt2$. For each regular symbol $D$, the separate fold-row condition at $e=1$ gives $q_1^*D(1)q_3=0$, whereas $q_1^*Aq_3=1$. Let $G=(T_A-T_D)F$. Then $G(1)=0$ and
\[
 \lim_{y\uparrow1}\frac{q_1^*G(y)}{(1-y)^\beta}=1.
\]
Thus the scalar H\"older seminorm and the componentwise maximum norm give
\[
 1\le[q_1^*G]_\beta\le\sqrt2\|G\|_{X_{\beta,4}}
 \le3\|T_A-T_D\|_{\mathcal B(X_{\beta,4})}.
\]
Taking the infimum over regular $T_D$ proves
$\operatorname{dist}_{\mathcal B(X_{\beta,4})}(T_A,\mathscr R_{\beta,4})\ge1/3$.
Since $A$ is constant, the obstruction is the mixing of distinct fold-difference directions, not a failure of H\"older regularity.
\end{proof}

The example also shows why the admissible space $\mathcal L_e$ in
the defect cannot be replaced by the larger space of matrices that
merely preserve $W_e$. Indeed, its matrix $A$ satisfies $AP_1=0$
and $A(W_1)\subset W_1$. If this weaker condition were used, one
could take $L=A$ and obtain zero local defect from
$B_1=h_1A$. The correct, fold-by-fold condition instead gives
$\Omega_1(B_1)=1/2$ and hence a positive norm distance. The
simultaneous-fold restriction is therefore needed in the
\emph{necessity} as well as in the construction of regular approximants.

The next corollary is a standard consequence of inverse-closedness of
the regular algebra and norm-continuity of inversion.

\begin{corollary}\label{inverseclosed}
The Banach algebra $\overline{\Rn}^{\,\|\cdot\|}$ is inverse-closed in $\mathcal B(A_\beta([0,1]))$. In particular, if $T\in\overline{\Rn}$ is invertible as a bounded operator, then $T^{-1}\in\overline{\Rn}$, and its spectrum computed in the closed algebra equals its spectrum in $\mathcal B(A_\beta)$.
\end{corollary}
\begin{proof}
Theorem 5.3 proves that $\Rn$ is inverse-closed in $\mathcal B(A_\beta)$. Let $T_m\in\Rn$ converge to an invertible $T\in\overline{\Rn}$. For sufficiently large $m$, the operators $T_m$ are invertible and $T_m^{-1}\in\Rn$. Continuity of inversion on the group of invertible bounded operators gives $T_m^{-1}\to T^{-1}$. Thus $T^{-1}\in\overline{\Rn}$. Replacing $T$ by $T-\lambda I$ gives equality of the two spectra.
\end{proof}

For symbols in the closure, the endpoint behaviour is better than (S1)--(S2) alone guarantee: the symbol extends continuously to the closed interval. This is used in Corollary~\ref{cor:new-closure-essential}.
\begin{lemma}\label{lem:closure-continuity}
Let $T_A\in\overline{\Rn}^{\|\cdot\|}$, with symbol $A$ as in Theorem~\ref{fullclosure}, and fix $e\in\{0,1\}$. Let $L_e\in\mathcal L_e$ attain $\Omega_e(B_e)=0$. Then $A(y)Q_e\to L_e$ as $y\to e$; in particular $L_e$ is unique, and $A$ extends continuously to $[0,1]$ with $A(e)=V_e(e)+L_e$. The endpoint matrix $A(e)$ satisfies the separate fold row-difference condition \eqref{regularsymbol}, and
\[
 Z_e(y):=h_e(y)(A(y)-A(e))\quad\text{satisfies}\quad
 \lim_{r\downarrow0}[Z_e]_{\beta,J_e(r)}=0.
\]
\end{lemma}
\begin{proof}
Put $P=P_e$, $Q=Q_e$, $h=h_e$, $V=AP$ and $B=hAQ$. By Theorem~\ref{fullclosure}, $V,B\in A_\beta(J_e)$, $B(e)=0$, and
\[
 \lim_{r\downarrow0}[B-hL_e]_{\beta,J_e(r)}=0.
\]
Since $B-hL_e$ vanishes at $e$, for $y\in J_e(r)$ we have
$\|B(y)-h(y)L_e\|\le h(y)[B-hL_e]_{\beta,J_e(r)}$.
Dividing by $h(y)$ proves $A(y)Q\to L_e$. Since $A(y)P=V(y)\to V(e)$, the continuous extension satisfies $A(e)=V(e)+L_e$. The rows of $V(e)$ coincide across each fold, since $V(e)=V(e)P$ and $V(e)W_e\subseteq W_e$, while $L_e\in\mathcal L_e$ satisfies the individual row conditions. Thus $A(e)$ satisfies \eqref{regularsymbol}. Finally $Z_e=h(V-V(e))+(B-hL_e)$; the second term has vanishing local seminorm, and $[h(V-V(e))]_{\beta,J_e(r)}\le2r^\beta[V]_{\beta,J_e}\to0$.
\end{proof}

\begin{theorem}
\label{thm:bilipschitz-full}
Let $u:[0,1]\to[0,1]$ be continuous. Suppose that
$0=s_0<s_1<\cdots<s_n=1$, that each restriction
$u|_{[s_{j-1},s_j]}$ is a homeomorphism onto $[0,1]$,
and that both this restriction and its inverse are Lipschitz.
Set $\tau_j=(u|_{[s_{j-1},s_j]})^{-1}$ and
$\phi_j=\tau_j\circ u$. Then the exact finite-fiber representation for all $A_\beta$
functions, the regular
fiber-matrix spectrum theorem, the full norm-closure criterion
of Theorem~\ref{fullclosure}, and Corollary~\ref{cor:operator-germ} and Lemma~\ref{lem:closure-continuity} hold for
\[
 \mathscr R_{\beta,u}
   =\left\{M_v+\sum_{j=1}^nM_{p_j}C_{\phi_j}:\
              v,p_j\in A_\beta([0,1])\right\}
\]
with the same endpoint fold conditions, using the branch points
$s_j$ instead of $j/n$. All constants may depend on the finite
collection of Lipschitz constants and on the branch lengths.
\end{theorem}
\begin{proof}
We first note that a continuous map $\psi:[0,1]\to[0,1]$
which is Lipschitz and monotone on finitely many intervals
induces a bounded composition operator on $A_\beta$.
The H\"older estimate follows from its Lipschitz constant.
For the restricted-variation condition, split admissible
intervals at the finitely many monotonicity breakpoints.
The resulting total $\beta$-power budget increases by a fixed
factor. On any one monotonicity interval, the images of
intervals with disjoint interiors again have disjoint interiors;
their total length and $\beta$-power sum are bounded by the
corresponding original quantities multiplied by the Lipschitz
constant and its $\beta$-power. Thus the restricted variation
tends to zero. This proof does not require the map to be affine.

Each $\tau_j$ is Lipschitz and monotone, and each $\phi_j$
is continuous, Lipschitz, and monotone on each source branch.
Hence all compositions used in the finite-fiber proof are bounded.
The map $Uf=(f\circ\tau_1,\ldots,f\circ\tau_n)$ is bounded;
its inverse glues $F_j\circ u$ on the original branches.
The gluing proof of Lemma~\ref{branchmodel} works with
bi-Lipschitz branches using the preceding composition estimate.
The endpoint matching spaces $W_0,W_1$ and the regular row-difference
conditions depend only on which adjacent branches meet at level
$0$ or $1$, not on linearity of the branch maps.
Every estimate in the proof of Theorem~\ref{fullclosure}
is made in the level coordinate $y$ and uses these matching
spaces and the Banach algebra $A_\beta([0,1])$.
It therefore applies without change after the branch isomorphism.
The endpoint arguments likewise take place entirely in the level coordinate.
\end{proof}

The hypotheses in Theorem~\ref{thm:bilipschitz-full}
are stronger than finite fibers alone; see
Example~\ref{ex:branch-regularity}.

\section{The full finite-branch commutant and its spectrum}
\label{sec:full-commutant}
Theorem~\ref{thm:exact-finite-fiber} and
Theorem~\ref{cor:full-symmetric} established exact locality and
the complete symmetric commutant in Sections~2--3. We now address
the distinct problem of characterizing every bounded commuting
operator for a finite full-branch multiplier. This requires the
weighted endpoint conditions below; pointwise finite-fiber locality
alone does not imply them.

For the next theorem let $u$ have the $n$ full bi-Lipschitz branches of Theorem~\ref{thm:bilipschitz-full}. Let $U$ and $X_{\beta,n}$ be as in Lemma~\ref{branchmodel}. The finite-dimensional spaces $W_e$, and projections $P_e,Q_e$, are those of Section~6. For $e\in\{0,1\}$ write $h_e(y)=|y-e|^\beta$. Choose fixed Lipschitz cutoffs $\chi_0,\chi_1$ equal to one near their respective endpoints and supported in disjoint endpoint neighborhoods. For a matrix $A:(0,1)\to M_n(\mathbb C)$, expressions such as $\chi_e A P_e\in M_n(A_\beta)$ mean that the corresponding product extends by zero to an $A_\beta$ matrix on $[0,1]$.

\begin{theorem}\label{thm:full-branch-commutant}
Let $u$ be as above. Under the branch isomorphism $U$, the full commutant $\{M_u\}'$ consists \emph{exactly} of the operators
\[
 T_A F(y)=A(y)F(y),\qquad 0<y<1,
\]
where the uniquely determined matrix $A:(0,1)\to M_n(\mathbb C)$ satisfies all three conditions:
\begin{enumerate}
\item[(C1)] Every entry of $A$ belongs locally to $A_\beta$ on $(0,1)$.
\item[(C2)] For $e=0,1$,
$\chi_e A P_e\in M_n(A_\beta)$, and its endpoint value sends $W_e$ into $W_e$:
\[
 Q_e(\chi_e A P_e)(e)=0.
\]
\item[(C3)] For $e=0,1$, $\chi_e h_e A Q_e\in M_n(A_\beta)$ and $(\chi_e h_e A Q_e)(e)=0$.
\end{enumerate}
Here $T_A$ is defined at each endpoint by its continuous extension to $X_{\beta,n}$.
\end{theorem}
\begin{proof}
The branch lemma gives $UM_uU^{-1}=M_y$ on $X_{\beta,n}$. Let $S\in\mathcal B(X_{\beta,n})$ commute with $M_y$. We show directly that this commutation implies fiber locality, including at the endpoints. Fix $y_0\in[0,1]$. If $H\in X_{\beta,n}$ vanishes on a neighborhood of $y_0$, then $H=(y-y_0)G$ for some $G\in X_{\beta,n}$: divide by $y-y_0$ away from that neighborhood and extend by zero near $y_0$. The scalar reciprocal is smooth on the support of $H$, and division preserves the endpoint matching conditions. Therefore
\[
 (SH)(y_0)=(S M_{y-y_0}G)(y_0)=(M_{y-y_0}SG)(y_0)=0.
\]
Now assume $F(y_0)=0$. Choose scalar Lipschitz cutoffs $\chi_{y_0,r}$ equal to one within distance $r$ of $y_0$, supported within distance $2r$, and with Lipschitz constant at most $C/r$ (one-sided at an endpoint). The componentwise uniform localization estimate gives
\begin{equation}\label{eq:full-fiber-uniform-localization}
 \|\chi_{y_0,r}F\|_{X_{\beta,n}}\le C_{\beta,n}\|F\|_{X_{\beta,n}},
\end{equation}
where the bound does not depend on $r$ or $y_0$. Since these cutoffs are scalar, they preserve the endpoint matching spaces. The function $(1-\chi_{y_0,r})F$ vanishes near $y_0$, so
$ (SF)(y_0)=(S(\chi_{y_0,r}F))(y_0). $
Choose $s_r=y_0+3r$ when $y_0<1$, and $s_r=1-3r$ when $y_0=1$, taking $r$ small enough that $s_r\in(0,1)$ and $s_r$ lies outside the support. The function $\chi_{y_0,r}F$ vanishes near $s_r$, hence $(S(\chi_{y_0,r}F))(s_r)=0$. Consequently
\[
 \|(SF)(y_0)\|_{\infty}\le [S(\chi_{y_0,r}F)]_{\beta}|y_0-s_r|^{\beta}
 \le C_{\beta,n}\|S\|\|F\|_{X_{\beta,n}}(3r)^\beta\longrightarrow0.
\]
In particular $F(e)=0$ implies $(SF)(e)=0$ for $e=0,1$. For every interior $y_0$, the same argument proves that $(SF)(y_0)$ depends only on $F(y_0)$. Evaluation $X_{\beta,n}\to\mathbb C^n$ is onto there, using smooth vectors supported away from the endpoints. Hence a unique matrix $A(y_0)$ satisfies $(SF)(y_0)=A(y_0)F(y_0)$ for all $F$.

For any relatively compact interval $J\subset(0,1)$, choose $\psi\in C^\infty_c((0,1))$ equal to one on $J$. The vectors $\psi\mathbf e_j$ belong to $X_{\beta,n}$, and $S(\psi\mathbf e_j)\in X_{\beta,n}$, proving (C1).

To obtain (C2), apply $S$ to the endpoint test vectors $F=\chi_e P_ez$, $z\in\mathbb C^n$. Their images are precisely $\chi_e A P_ez$ in the interior, and the endpoint values lie in $W_e$. For (C3), apply $S$ to $F=\chi_e h_e Q_ez$. This vector vanishes at $e$, so the corresponding output $\chi_e h_e A Q_ez$ belongs to $X_{\beta,n}$ and vanishes there. Thus (C1)--(C3) are necessary.

Conversely, assume (C1)--(C3). Near an endpoint $e$ write $F=P_eF+Q_eF$. The first contribution $A P_eF$ belongs to $A_\beta^n$ by (C2). The second is
\[
 A Q_eF=\frac{(h_e A Q_e)(Q_eF)}{h_e}.
\]
Both factors in the numerator belong to $A_\beta$ and vanish at $e$. The weighted multiplier estimate \eqref{weighted} therefore implies that this vector belongs to $A_\beta^n$ and has endpoint value zero. Thus the result belongs to $X_{\beta,n}$ and its endpoint value lies in $W_e$ by (C2). Condition (C1) takes care of compact subintervals of $(0,1)$; a finite partition of unity proves $T_AF\in X_{\beta,n}$ globally. The graph of $T_A$ is closed: convergence in $X_{\beta,n}$ implies pointwise convergence and $A(y)$ is fixed at each interior point. The closed graph theorem gives boundedness. Finally $T_A$ commutes with $M_y$, hence $U^{-1}T_AU$ commutes with $M_u$.
\end{proof}

The spectral formula for this \emph{full} commutant requires the closure of the interior fiber spectra, since the transverse entries of $A(y)$ need not have endpoint limits.
\begin{lemma}\label{lem:full-weighted-inverse}
Let $J=[0,\eta]$, $h(y)=y^\beta$, and
\[
 \mathcal S_h=\{a:J\setminus\{0\}\to\mathbb C:\ ha\in A_\beta(J),\ (ha)(0)=0\}.
\]
Then $\mathcal S_h$ is an algebra of bounded functions. If $a\in\mathcal S_h$ and $\inf_{0<y\le\eta}|a(y)|>0$, then $a^{-1}\in\mathcal S_h$. Moreover, if $a\in\mathcal S_h$ and $g\in A_\beta(J)$ satisfies $g(0)=0$, then $ag\in A_\beta(J)$ and $(ag)(0)=0$. The analogous assertions hold at the right endpoint.
\end{lemma}
\begin{proof}
Put $b=ha\in A_\beta(J)$, $b(0)=0$. For $y>0$,
$|a(y)|=|b(y)|/h(y)\le[b]_\beta$. Thus every element of
$\mathcal S_h$ is bounded. If $a,k\in\mathcal S_h$, then
\[
 h(ak)=\frac{(ha)(hk)}{h}\in A_\beta(J),\qquad
 (h ak)(0)=0,
\]
by the weighted multiplier estimate \eqref{weighted} (after moving
the distinguished point to the endpoint). Constants also belong
to $\mathcal S_h$, since $h\in A_\beta(J)$.

For inverse-closedness, suppose $|a(y)|\ge m>0$ and define
$c(0)=0$ and $c(y)=h(y)/a(y)=h(y)^2/b(y)$ for $y>0$.
For $0<x<y\le\eta$, expand
\[
 c(y)-c(x)=\frac{h(y)^2-h(x)^2}{b(y)}
  +\frac{h(x)^2\bigl(b(x)-b(y)\bigr)}{b(y)b(x)}.
\]
The bounds $|b(z)|\ge m h(z)$ and $h(x)\le h(y)$ give
\begin{equation}\label{eq:weighted-inverse-increments}
 |c(y)-c(x)|\le \frac{2}{m}|h(y)-h(x)|
               +\frac{1}{m^2}|b(y)-b(x)|.
\end{equation}
The same inequality holds for $x=0$. It proves the
$\beta$-H\"older estimate for $c$. More importantly, summing
\eqref{eq:weighted-inverse-increments} over any admissible disjoint
interval family yields
\[
 V_{\beta,\delta,M}(c)
 \le \frac{2}{m}V_{\beta,\delta,M}(h)
   +\frac{1}{m^2}V_{\beta,\delta,M}(b)\longrightarrow0
 \qquad(\delta\downarrow0).
\]
Hence $c\in A_\beta(J)$ and $c(0)=0$, which is exactly
$a^{-1}=c/h\in\mathcal S_h$.

Finally, if $g\in A_\beta(J)$ and $g(0)=0$, then
$ag=(ha)g/h\in A_\beta(J)$ and $(ag)(0)=0$ by the same
weighted multiplier estimate. Reflection gives the statement
at the right endpoint.
\end{proof}

\begin{theorem}\label{thm:full-commutant-spectrum}
Let $u$ have the $n$ full bi-Lipschitz branches described above, and let $T=U^{-1}T_AU\in\{M_u\}'$ be represented by the unique symbol of Theorem~\ref{thm:full-branch-commutant}. Then
\begin{equation}\label{eq:full-commutant-spectrum}
 \boxed{\quad\sigma_{\mathcal B(A_\beta)}(T)
       =\overline{\bigcup_{0<y<1}\sigma_{M_n(\mathbb C)}(A(y))}.\quad}
\end{equation}
If $\lambda$ is outside this closed set, the inverse of $T-\lambda I$ is the commuting operator represented by $(A(y)-\lambda I)^{-1}$.
\end{theorem}
\begin{proof}
Fix an interior level $y_0\in(0,1)$. The evaluation map
$E_{y_0}:X_{\beta,n}\to\mathbb C^n$ is onto: a smooth vector-valued
function supported near $y_0$ realizes any prescribed value.
Moreover,
\[
 E_{y_0}(T_A-\lambda I)
       =(A(y_0)-\lambda I)E_{y_0}.
\]
If $T_A-\lambda I$ were invertible, its surjectivity and the
surjectivity of $E_{y_0}$ would force $A(y_0)-\lambda I$ to be
onto, hence invertible. Therefore every interior fiber eigenvalue
belongs to $\sigma(T)$, and so does the closure of their union.

Conversely, suppose $\lambda$ lies outside that closed union.
Set $B(y)=A(y)-\lambda I$ and
\[
 \varepsilon=\operatorname{dist}\!\left(
 \lambda,\overline{\bigcup_{0<y<1}\sigma(A(y))}\right)>0.
\]
Conditions (C1)--(C3) imply that $B$ is uniformly bounded.
Indeed, $BP_e$ is bounded near $e$, and because
$(h_e BQ_e)(e)=0$ and $h_e(y)=|y-e|^\beta$,
\[
 \|B(y)Q_e\|\le C_n [h_e BQ_e]_{\beta,J_e}
 \quad\text{for }y\in J_e\setminus\{e\}.
\]
Condition (C1) gives boundedness on the remaining compact interval.
Thus $K:=\sup_{0<y<1}\|B(y)\|$ is finite.
All eigenvalues of $B(y)$ have modulus at least $\varepsilon$.
Consequently
\begin{equation}\label{eq:full-commutant-uniform-inverse}
 |\det B(y)|\ge\varepsilon^n,
 \qquad \sup_{0<y<1}\|B(y)^{-1}\|
 \le C_n K^{n-1}\varepsilon^{-n}<\infty,
\end{equation}
where the second estimate follows from the adjugate formula
and $n$ is fixed. Write $D(y)=B(y)^{-1}$.

Fix either endpoint $e$, and work on a sufficiently short closed
endpoint interval $J_e$ where the cutoffs in (C2)--(C3) equal one.
For brevity put $P=P_e$, $Q=Q_e$, and $h=h_e$.
By (C2), $BP$ has entries in $A_\beta(J_e)$; by (C3),
$hBQ$ has entries in $A_\beta(J_e)$ and vanishes at $e$.
Every entry of $BP$ belongs to $\mathcal S_h$, since multiplication
by $h$ preserves $A_\beta$ and $h(e)=0$. Thus every entry of
$B=BP+BQ$ belongs to $\mathcal S_h$.
As $\mathcal S_h$ is a unital algebra, the entries of
$\operatorname{adj}B$ and $\det B$ belong to $\mathcal S_h$.
The determinant bound in \eqref{eq:full-commutant-uniform-inverse}
and Lemma~\ref{lem:full-weighted-inverse} imply
$(\det B)^{-1}\in\mathcal S_h$. Hence the adjugate formula gives
$D\in M_n(\mathcal S_h)$, which is precisely (C3) for $D$:
$hDQ\in M_n(A_\beta(J_e))$ and $(hDQ)(e)=0$.
Ordinary matrix inversion on compact subintervals of $(0,1)$
also proves (C1) for $D$.

The remaining condition (C2) is \emph{stronger} than weighted
membership: we must show $DP\in M_n(A_\beta(J_e))$ and
$Q(DP)(e)=0$. Let $a_e:W_e\to W_e$ be the endpoint map of $BP$.
It is invertible. In fact, if $w\in W_e$, then
\[
 \|B(y)w\|\ge
 \left(\sup_{0<z<1}\|D(z)\|\right)^{-1}\|w\|.
\]
Taking $y\to e$ shows the same lower bound for $\|a_ew\|$;
since $W_e$ is finite-dimensional, $a_e$ is onto and invertible.

Define $B_0(y)=B(y)P+Q$ on $J_e$. Its entries belong to
$A_\beta(J_e)$, and $B_0(e)$ is block diagonal relative to
$W_e\oplus W_e^\perp$, with diagonal blocks $a_e$ and
$I_{W_e^\perp}$. After shortening $J_e$ if necessary,
$B_0(y)$ is invertible there, and Lemma~\ref{lem:composition-reciprocal} gives
$B_0^{-1}\in M_n(A_\beta(J_e))$.
Put $C=B_0^{-1}B$. Then
\[
 CP=P,\qquad C^{-1}P=P,\qquad C^{-1}=DB_0
  \in M_n(\mathcal S_h).
\]
The first two identities express the fact that $B$ and $B_0$
agree on $W_e$. Decomposing $B_0^{-1}P$ along $P+Q$ gives
\begin{equation}\label{eq:full-commutant-DP}
 DP=P B_0^{-1}P+(C^{-1}Q)(Q B_0^{-1}P).
\end{equation}
Here $G:=QB_0^{-1}P$ belongs to $M_n(A_\beta(J_e))$ and
$G(e)=0$, since $B_0(e)^{-1}$ preserves $W_e$. Also
$hC^{-1}Q\in M_n(A_\beta(J_e))$ and vanishes at $e$.
Applying the \emph{product-with-a-vanishing-function} part of
Lemma~\ref{lem:full-weighted-inverse} entrywise shows that
$(C^{-1}Q)G\in M_n(A_\beta(J_e))$ with zero endpoint value.
Equation \eqref{eq:full-commutant-DP} now proves
$DP\in M_n(A_\beta(J_e))$ and
$Q(DP)(e)=0$. Thus (C2) holds for $D$ at both endpoints.

By Theorem~\ref{thm:full-branch-commutant}, $T_D$ is a bounded
operator on $X_{\beta,n}$ commuting with $M_y$.
At every interior point we have $D(y)B(y)=B(y)D(y)=I$;
therefore $T_DT_B=T_BT_D=I$ at interior points and, by
continuity, also at the endpoints. Hence $T_D$ is the bounded
two-sided inverse of $T_B$, and $U^{-1}T_DU$ is the inverse of
$T-\lambda I$ on $A_\beta$. This completes the reverse inclusion.
\end{proof}

\begin{corollary}
\label{cor:full-branch-essential}
Under the finite full bi-Lipschitz branch assumptions of
Theorem~\ref{thm:full-branch-commutant}, let
$T=U^{-1}T_AU\in\{M_u\}'$ be any bounded commuting operator,
not necessarily in the regular operator-norm closure. Then
\begin{equation}\label{eq:full-branch-essential}
\boxed{\sigma_{\mathrm{ap}}(T)=\sigma_{\mathrm{usf}}(T)
=\sigma_{\mathrm{ess}}(T)=\sigma(T)}
=\overline{\bigcup_{0<y<1}\sigma(A(y))}.
\end{equation}
In particular, $T$ is Fredholm if and only if it is invertible.
\end{corollary}
\begin{proof}
Theorem~\ref{thm:full-commutant-spectrum} gives the last equality
for the ordinary spectrum. We prove the reverse inclusion for
$\sigma_{\mathrm{usf}}(T)$ and $\sigma_{\mathrm{ap}}(T)$.
Fix $y_0\in(0,1)$ and $\lambda\in\sigma(A(y_0))$. Choose
$0\ne\xi\in\ker(A(y_0)-\lambda I)$ and a nonzero
$\eta\in C_c^\infty((-1,1))$. For sufficiently small $r>0$ let
\[
 F_r(y)=r^\beta\eta\!\left(\frac{y-y_0}{r}\right)\xi.
\]
These functions belong to $X_{\beta,n}$, since they vanish near
both endpoints. Their $X_{\beta,n}$ norms are bounded above and away
from zero, whereas $\|F_r\|_\infty\to0$. By (C1), $A$ is
$\beta$-H\"older on a fixed neighborhood of $y_0$.
Lemma~\ref{lem:localized-spectral-test} applies. Using
$(A(y_0)-\lambda I)\xi=0$ and the H\"older product estimate,
\[
 \|(T_A-\lambda I)F_r\|_{X_{\beta,n}}=O(r^\beta).
\]
After normalization, the $F_r$ form a sequence of unit vectors
with no norm-convergent subsequence. Thus $\lambda$ belongs to the
approximate point spectrum. It also belongs to the upper
semi-Fredholm spectrum: otherwise the open mapping theorem applied to the quotient by the kernel gives
$\operatorname{dist}(F_r,\ker(T_A-\lambda I))\le
C\|(T_A-\lambda I)F_r\|$ and finite-dimensionality of the kernel
would give a convergent subsequence. Both spectra are closed, so
their inclusions extend to the closure of the union of all interior
fiber eigenvalues. The remaining inclusions are standard, and the
ordinary spectral equality completes the proof.
\end{proof}

\begin{corollary}\label{cor:class-full}
Under the hypotheses of Theorem~\ref{thm:full-branch-commutant} (in particular for $u=u_n$), the class $\mathfrak M_{\beta,n}$ of Corollary~\ref{cor:operator-germ} coincides with the full commutant under $U$:
\[
 \mathfrak M_{\beta,n}=U\{M_u\}'U^{-1}.
\]
Consequently the quotient of the full commutant by the operator-norm closure of the regular algebra is isomorphic, as a Banach space, to a closed subspace of the direct sum of the two germ quotients via the map of Corollary~\ref{cor:operator-germ}.
\end{corollary}
\begin{proof}
Condition (S1) is (C1). By (C1), membership of $AP_e$ and $h_eAQ_e$ in $A_\beta$ near $e$ does not depend on the endpoint interval or cutoff. Since $V_e(e)=V_e(e)P_e$, the inclusion $V_e(e)W_e\subseteq W_e$ is equivalent to $Q_eV_e(e)=0$. Thus (S2) is equivalent to (C2) together with (C3). Apply Theorem~\ref{thm:full-branch-commutant}, Corollary~\ref{cor:operator-germ}, and Theorem~\ref{thm:bilipschitz-full}.
\end{proof}

For the regular closure, the closure bar in the spectral formula can be removed, because the symbol is continuous on the closed interval.
\begin{corollary}\label{cor:new-closure-essential}
Let $u$ have the full bi-Lipschitz branches of Theorem~\ref{thm:full-branch-commutant}, and let $T_A$ belong to the operator-norm closure of its regular finite-branch algebra. Extend its symbol $A$ continuously to $[0,1]$ by Lemma~\ref{lem:closure-continuity}. Then
\[
 \sigma_{\mathrm{ap}}(T_A)=\sigma_{\mathrm{usf}}(T_A)=\sigma_{\mathrm{ess}}(T_A)=\sigma(T_A)
 =\bigcup_{0\le y\le1}\sigma(A(y)).
\]
Consequently $T_A$ is Fredholm if and only if it is invertible, equivalently if and only if $\det A(y)\ne0$ for every $y\in[0,1]$. In that case $T_A^{-1}=T_{A^{-1}}$ belongs to the regular operator-norm closure.
\end{corollary}
\begin{proof}
By Corollary~\ref{cor:class-full}, $T_A$ is in the full commutant, so Corollary~\ref{cor:full-branch-essential} applies. Continuity of $A$ on compact $[0,1]$ and continuity of finite-matrix spectra show that the union over $[0,1]$ is closed and equals the closure of the union over $(0,1)$. Inverse-closedness of the regular closure follows from Corollary~\ref{inverseclosed}; the inverse is represented by $A^{-1}$ by Theorem~\ref{thm:full-commutant-spectrum}.
\end{proof}

\noindent\label{rem:finite-scope}
Theorem~\ref{thm:exact-finite-fiber} holds for every real $u\in A_\beta$ with finite level sets, even when no inverse branch is Lipschitz. It gives a complete pointwise representation of all bounded commuting operators, but does not by itself yield the finite-dimensional endpoint conditions (C1)--(C3) or a fiber-matrix spectral formula for an arbitrary oscillatory finite-to-one map. Theorems~\ref{thm:full-branch-commutant} and \ref{thm:full-commutant-spectrum} supply those stronger results under the explicitly stated finite full-branch hypotheses. In particular, Theorem~\ref{cor:full-symmetric} settles the entire commutant of the original symmetric multiplier. Khani Robati and Seddighi~\cite{KhaniRobatiSeddighi2000} already studied operator-matrix commutants of direct sums and multipliers of Lipschitz algebras. The weighted endpoint conditions, norm-closure criteria, and spectral conclusions are established here under the stated restricted-variation hypotheses.

\medskip
\noindent\textbf{Authors' contact information}

\smallskip
\noindent H\"ulya Duru\\
Department of Mathematics, Faculty of Science, Istanbul University,\\
Vezneciler-Istanbul, 34134, Turkey\\
\texttt{hduru@istanbul.edu.tr}

\smallskip
\noindent Serkan \.{I}lter\\
Department of Mathematics, Faculty of Science, Istanbul University,\\
Vezneciler-Istanbul, 34134, Turkey\\
\texttt{ilters@istanbul.edu.tr}

\smallskip
\noindent Arkady Kitover\\
Community College of Philadelphia, 1700 Spring Garden St.,\\
Philadelphia, PA 19130, USA\\
\texttt{akitover@ccp.edu}
\end{document}